\documentclass[10pt, conference, letterpaper]{IEEEtran}

\usepackage{graphicx}				
														
\usepackage{amssymb,amsmath,amsthm}
\usepackage{enumitem}
\usepackage{color}

\usepackage{booktabs}
\usepackage{algorithm}

\newtheorem*{Model}{Model}

\newtheorem{Assumption}{Assumption}

\newtheorem{Thm}{Theorem}
\newtheorem{Fact}{Fact}
\newtheorem{Lem}{Lemma}
\newtheorem{Cor}{Corollary}

\DeclareMathOperator*{\argmin}{argmin}

\newcommand*{\tran}{^{\mkern-1.5mu\mathsf{T}}}

\newcommand{\script}[1]{{{\cal{#1} }}}
\newcommand{\norm}[1]{\Vert{#1}\Vert}
\newcommand{\expect}[1]{\mathbb{E}\left[#1\right]}

\usepackage{url,fancyhdr}
\IEEEoverridecommandlockouts

\title{Learning Aided Optimization for Energy Harvesting Devices with Outdated State Information}

\author{Hao Yu and Michael J. Neely 
\thanks{Hao Yu is with DAMO academy at Alibaba Group (US) and was with the Electrical Engineering department at the University of Southern California. Michael J. Neely is with the Electrical Engineering department at the University of Southern California. This work is supported in part by grant NSF CCF-1718477.} 
\thanks{This work was presented in part at IEEE International Conference on Computer Communications (INFOCOM), Honolulu, USA, Apr. 2018 \cite{YuNeely18INFOCOM} }}

\begin{document}
\maketitle
\thispagestyle{fancy}

\begin{abstract}
This paper considers utility optimal power control for energy harvesting wireless devices with a finite capacity battery. The distribution information of the underlying wireless environment and harvestable energy is unknown and only outdated system state information is known at the device controller. This scenario shares similarity with  Lyapunov opportunistic optimization and online learning but is different from both. By a novel combination of Zinkevich's online gradient learning technique and the drift-plus-penalty technique from Lyapunov opportunistic optimization, this paper proposes a learning-aided  algorithm that achieves utility within $O(\epsilon)$ of the optimal, for any desired $\epsilon>0$, by using a battery with an $O(1/\epsilon)$ capacity. The proposed algorithm has low complexity and makes power investment decisions based on system history, without requiring knowledge of the system state or its probability distribution. 
\end{abstract}

\begin{IEEEkeywords}
online learning; stochastic optimization; energy harvesting
\end{IEEEkeywords}

\section{Introduction}
Energy harvesting can enable self-sustainable and perpetual wireless devices. 
By harvesting energy from the environment and storing it in a battery for future use, we can significantly improve energy efficiency and device lifetime.  Harvested energy can come from solar, wind, vibrational, thermal, or even radio sources \cite{Paradiso05,Sudevalayam11,Ulukus15JSAC}. Energy harvesting has been identified as a key technology for wireless sensor networks \cite{Kansal07TECS}, internet of things (IoT) \cite{Kamalinejad15}, and 5G communication networks \cite{Hossian15}.  However, 
the development of harvesting algorithms is complex because the harvested energy is highly dynamic and the device environment and energy needs are also dynamic.  Efficient algorithms should learn when to take energy from the battery to power device tasks that bring high utility, and when to save energy for future use.

There have been large amounts of work developing efficient power control policies to maximize the utility of energy harvesting devices. Most of existing literature assumes the current system state is observable \cite{Ulukus12TCOM,Yener12TWC, Blasco13TWC,Michelusi13TCOM, Shaviv16JSAC,Wu17TMC, Arafa17ISIT}. In the highly ideal case where the future system state (both the wireless channel sate and energy harvesting state) can be perfectly predicted, optimal power control strategies that maximize the throughput of wireless systems are considered in \cite{Ulukus12TCOM,Yener12TWC}. Dynamic power policies are developed in \cite{Shaviv16JSAC} and \cite{Arafa17ISIT} for energy harvesting scenarios with a fixed known utility function and i.i.d. energy arrivals. In a more realistic case with only the statistics and causal knowledge of the system state, power control policies based on Markov Decision Processes (MDP) are considered in \cite{Blasco13TWC,Michelusi13TCOM}.  The work \cite{Michelusi13TCOM} also develops reinforcement learning based approach to address more challenging scenarios with observable current system state but unknown statistical knowledge. However, the reinforcement learning based method is restricted to problems with finite system states and power actions, and discounted long term utilities. For the case with unknown statistical knowledge but observable current system state, work \cite{Wu17TMC} develops suboptimal power control policies based on approximation algorithms.

However, there is little work on the challenging scenario where neither the distribution information nor the system state information are known. In practice, the amount of harvested energy on each slot is known to us only after it arrives and is stored into the battery. Further, the wireless environment is often unknown before the power action is chosen. For example, the wireless channel state in a communication link is measured at the receiver side and then reported back to the transmitter with a time delay. If the fading channel varies very fast, the channel state feedback received at the transmitter can be outdated.  Another example is power control for sensor nodes that detect unknown targets where the state of targets is known only after the sensing action is performed. 

In this paper, we consider utility-optimal power control in an energy harvesting wireless device with outdated state information and unknown state distribution information. This problem setup is closely related to but different from the Lyapunov opportunistic power control considered in works \cite{Gatzianas10TWC,HuangNeely13TON,UrgaonkarNeely11Sigmetrics} with instantaneous wireless channel state information.  The policies developed in \cite{Gatzianas10TWC,HuangNeely13TON,UrgaonkarNeely11Sigmetrics} are allowed to adapt their power actions to the instantaneous system states on each slot, which are unavailable in our problem setup.  The problem setup in this paper is also closely related to online convex optimization where control actions are performed without knowing instantaneous system states \cite{Zinkevich03ICML,book_PredictionLearningGames,Shalev-Shwartz11FoundationTrends}. However, classical methods for online convex learning require the control actions to be chosen from a simple fixed set.   Recent developments on online convex learning with constraints either assume the constraints are known long term constraints or yield constraint violations that eventually grow to infinity \cite{Mahdavi12JMLR, Jenatton16ICML, YuNeely16ArxivOnlineOpt, NeelyYu17arXiv, YuNeely17NIPS}.
These methods do not apply to our problem since the power to be used can only be drained from the finite capacity battery whose backlog is time-varying and depends on previous actions.

By combining the drift-plus-penalty (DPP) technique for Lyapunov opportunistic optimization \cite{book_Neely10} and the online gradient learning technique for online convex optimization \cite{Zinkevich03ICML}, we develop a novel learning aided dynamic power control algorithm that can achieve an $O(\epsilon)$ optimal utility by using a battery with an $O(1/\epsilon)$ capacity for energy harvesting wireless devices with outdated state information.  The first part of this paper treats a system with independent and identically distributed (i.i.d.) states. Section \ref{section:non-iid} extends to non-i.i.d. cases that are not considered in our conference version \cite{YuNeely18INFOCOM}. The notations used in this paper are summarized in Table \ref{table:notation}.

\begin{table}[th]
	\label{table:notation}
	\begin{center}
		\begin{small}
				\begin{tabular}{ll}
					\toprule
					{\footnotesize \bf  Notation}  & {\footnotesize \bf Definition} \\
					\midrule
					$t$    & slot index  \\
					$e[t]$ & energy arrival on slot $t$\\
					$\mathbf{s}[t]$ & system state on slot $t$ \\
					$\omega[t]$ & $\omega[t] = (e[t], \mathbf{s}[t])$\\
					$\mathbf{p}[t]$ & $n$-dimensional power allocation vector on slot $t$\\
					$p^{\max}$ & maximum total power that can be used  \\ &per slot (see Assumption \ref{ass:basic})\\
					$e^{\max}$ & maximum energy arrival per slot (see Assumption \ref{ass:basic})\\
					$D_i$ &  subgradient bound of the $i$-th coordinate  \\ &of the utility  function (see Assumption \ref{ass:basic})\\
					$D$ & $D = \sqrt{\sum_{i=1}^n D_i^2}$ (see Assumption \ref{ass:basic})\\
					$D^{\max}$ & $\max\{D_1, \ldots, D_n\}$ (see Corollary \ref{cor:large-Q-decrease-p})\\
					$B$ & $B=(\max\{e^{\max}, p^{\max}\})^2$ (see Lemma \ref{lm:drift})\\
					$E^{\max}$ & battery capacity\\
					$U^{\ast}$ & utility bound (see Lemma \ref{lm:utility-upper-bound})\\
					$V$ & algorithm parameter (see Algorithm \ref{alg:new-alg})\\
					$Q[t]$ & algorithm parameter (see Algorithm \ref{alg:new-alg})\\
					$Q^{l}$ & $Q^{l} = \lceil V \rceil (D^{\max} + 2p^{\max} + e^{\max})$ (see Theorem \ref{thm:Q-bound})\\
					\bottomrule
				\end{tabular}
		\end{small}
	\end{center}
	\caption{A summary of notations used in this paper}
\end{table}

\section{Problem Formulation}

Consider an energy harvesting wireless device that operates in normalized time slots $t\in\{1,2,\ldots\}$. Let $\omega[t] = (e[t],\mathbf{s}[t]) \in \Omega$ represent the system state on each slot $t$, where 
\begin{itemize} 
\item $e[t]$ is the amount of harvested energy for slot $t$ (for example, through solar, wind, radio signal, and so on). 
\item $\mathbf{s}[t]$ is the wireless device state on slot $t$ (such as the vector of channel conditions over multiple subbands).  
\item $\Omega$ is the state space for all $\omega[t] =(e[t],\mathbf{s}[t])$ states. 
\end{itemize} 
Assume $\{\omega[t]\}_{t=1}^{\infty}$ evolves in an independent and identically distributed (i.i.d.) manner according to an unknown distribution. Further, the state $\omega[t]$ is \emph{unknown} to the 
device \emph{until the end of slot $t$}. The device is powered by a finite-size battery. At the beginning of each slot $t\in\{1,2,\ldots\}$, the device 
draws energy from the battery and allocates it as an $n$-dimensional power decision vector $\mathbf{p}[t] = [p_1[t], \ldots, p_n[t]]\tran \in \mathcal{P}$ where $\mathcal{P}$ is a compact convex set  given by 
\begin{align*}
\mathcal{P} = \{\mathbf{p}\in \mathbb{R}^n: \sum_{i=1}^n p_i \leq p^{\max}, p_i \geq 0, \forall i\in
\{1,2,\ldots,n\}\}.
\end{align*}
Note that $n\geq1$ allows for power allocation over multiple orthogonal subbands and $p^{\max}$ is a given positive constant (restricted by hardware) and represents the maximum total power that can be used on each slot. The device receives a corresponding utility $U(\mathbf{p}[t]; \omega[t])$.  Since $\mathbf{p}[t]$ is chosen without knowledge of $\omega[t]$, the achieved utility is unknown until the end of slot $t$. 
 For each $\omega \in \Omega$, the utility function $U(\mathbf{p}; \omega)$ is assumed to be continuous and concave over $\mathbf{p} \in \mathcal{P}$. 
 An example is: 
 \begin{equation} \label{eq:log-utility}  
 U(\mathbf{p}; \omega) = \sum_{i=1}^n\log(1 + p_i[t]s_i[t]) 
 \end{equation} 
where $\mathbf{s}[t]=(s_1[t], \ldots, s_n[t])$ is the vector of (unknown) channel conditions over $n$ orthogonal subbands available to the wireless device. In this example, $p_i[t]$ represents the amount of power invested over subband $i$ in a rateless coding transmission scenario, and $U(p[t]; \omega[t])$ is the total throughput achieved on slot $t$.  We focus on fast time-varying wireless channels, e.g., communication scenarios with high mobility transceivers, where $\mathbf{s}[t]$ known at the transmitter is outdated since $\mathbf{s}[t]$ must be measured at the receiver side and then reported back to the transmitter with a time delay.
 
 \subsection{Further examples} 
 
 The above formulation admits a variety of other useful application scenarios.  For example, 
 it can be used to treat power control in cognitive radio systems.  Suppose an energy limited
 secondary user harvests energy and  operates over licensed spectrum occupied by primary users. In this case, $\mathbf{s}[t]=(s_1[t], \ldots, s_n[t])$ represents the channel activity of primary users over each subband. Since primary users are not controlled by the secondary user, $\mathbf{s}[t]$ is only known to the secondary user at the end of slot $t$.

Another application is a wireless sensor system. Consider an energy harvesting sensor node that collects information by detecting an unpredictable target. In this case, $\mathbf{s}[t]$ can be the state or action of the target on slot $t$. By using $\mathbf{p}[t]$ power  for signaling and sensing, we receive utility $U(\mathbf{p}[t]; \omega[t])$, which depends on state $\omega[t]$. For example, in a monitoring system, if the monitored target performs an action $\mathbf{s}[t]$ that we are not interested in, then the reward $U(\mathbf{p}[t]; \omega[t])$ by using $\mathbf{p}[t]$ is small.  Note that $\mathbf{s}[t]$ is typically unknown to us at the beginning of slot $t$ and is only disclosed to us at the end of slot $t$.

\subsection{Basic assumption} 
\begin{Assumption} \label{ass:basic}~
\begin{itemize}
\item  There exist a constant $e^{\max}>0$ such that $0\leq e[t] \leq e^{\max}, \forall t\in\{1,2,\ldots\}$.  
\item  Let $\nabla_\mathbf{p} U(\mathbf{p}; \omega)$ denote a subgradient (or gradient if $U(\mathbf{p}; \omega)$ is differentiable) vector of $U(\mathbf{p};\omega)$ with respect to $\mathbf{p}$ and let $\frac{\partial }{\partial p_i}U(\mathbf{p}; \omega), \forall i\in\{1,2,\ldots,n\}$ denote each component of vector $\nabla_\mathbf{p}U(\mathbf{p}; \omega)$.  There exist positive constants $D_1, \ldots, D_n$ such that  $\vert \frac{\partial }{\partial p_i}U(\mathbf{p}; \omega)\vert \leq D_i, \forall i\in\{1,2,\dots, n\}$ for all $\omega \in \Omega$ and all $\mathbf{p}\in \mathcal{P}$. This further implies there exists $D>0$, e.g., $D = \sqrt{\sum_{i=1}^n D_i^2}$, such that $ \norm{ \nabla_\mathbf{p} U (\mathbf{p}; \omega) }\leq D$ for all $\omega \in \Omega$ and all $\mathbf{p}\in \mathcal{P}$, where $\norm{\mathbf{x}}=\sqrt{\sum_{i=1}^nx_i^2}$ is the standard $l_2$ norm. 
\end{itemize}
\end{Assumption}

Such constants  $D_1, \ldots, D_n$ exist in most cases of interest, such as for utility functions \eqref{eq:log-utility}  with bounded $s_i[t]$ values. \footnote{This is always true when $s_i[t]$ are wireless signal strength attenuations.} 
The following fact follows directly from Assumption \ref{ass:basic}. Note that if $U(\mathbf{p}, \omega)$ is differentiable with respect to $\mathbf{p}$, then Fact \ref{fact:basic} holds even for non-concave $U(\mathbf{p}, \omega)$.
\begin{Fact}\label{fact:basic} [Lemma 2.6 in \cite{Shalev-Shwartz11FoundationTrends}]
For each $\omega \in \Omega$, $U(\mathbf{p};\omega)$ is $D$-Lipschitz over $\mathbf{p} \in \mathcal{P}$, i.e., 
$|U(\mathbf{p}_1; \omega) - U(\mathbf{p}_2; \omega)| \leq D\norm{\mathbf{p}_1-\mathbf{p}_2}, \forall \mathbf{p}_1, \mathbf{p}_2 \in \mathcal{P}. $
\end{Fact}

\subsection{Power control and energy queue model}

The finite size battery can be considered as backlog in an \emph{energy queue}.  Let $E[0]$ be the initial energy backlog in the battery and $E[t]$ be the energy stored in the battery at the {\bf end} of slot $t$. The power vector $\mathbf{p}[t]$ 
must satisfy the following \emph{energy availability constraint}:
\begin{align}
\mbox{$\sum_{i=1}^n$} p_i[t] \leq E[t-1], \forall t\in\{1,2,\ldots\}. \label{eq:energy-availability}
\end{align}
which requires the consumed power to be no more than what is available in the battery.

Let $E^{\max}$ be the maximum capacity of the battery. If the energy availability constraint \eqref{eq:energy-availability} is satisfied on each slot, the energy queue backlog $E[t]$ evolves as follows:
\begin{align}
E[t] =\min\{E[t-1] - \mbox{$\sum_{i=1}^n$} p_i[t] + e[t], E^{\max}\}, \forall t.  \label{eq:energy-queue-dynamic}
\end{align}

\subsection{An upper bound problem}

Let $\omega[t] = (e[t], s[t])$ be the random state vector on slot $t$.  Let $\expect{e} = \expect{e[t]}$ denote the expected 
amount of new energy that arrives in one slot. Define a function $h:\mathcal{P}\rightarrow\mathbb{R}$ by 
$$ h(\mathbf{p}) = \expect{U(\mathbf{p};\omega[t])}. $$
Since $U(\mathbf{p};\omega)$ is concave in $\mathbf{p}$ for all $\omega$ by Assumption \ref{ass:basic} and is $D$-Lipschitz over $p \in \mathcal{P}$ for all $\omega$ by Fact \ref{fact:basic}, we know $h(\mathbf{p})$ is concave and continuous.

The function $h$ is typically unknown because the distribution of $\omega$ is unknown. 
However, to establish a fundamental bound, suppose both $h$ and $\mathbb{E}[e]$ are known and 
consider choosing a fixed vector $\mathbf{p}$ to solve the following deterministic problem: 
\begin{align}
\max_{\mathbf{p}}~&  h(\mathbf{p}) \label{eq:deterministic-obj}\\
\text{s.t.}~~~&  \sum_{i=1}^n  p_i - \mathbb{E}[e] \leq 0 \label{eq:deterministic-energy-balance}\\
& \mathbf{p}\in \script{P} \label{eq:deterministic-box-cons}
\end{align}  
where  constraint \eqref{eq:deterministic-energy-balance} requires that the consumed energy is no more than $\mathbb{E}[e]$. 

Let $\mathbf{p}^\ast$ be an optimal solution of problem \eqref{eq:deterministic-obj}-\eqref{eq:deterministic-box-cons} and $U^\ast$ be its corresponding utility value of \eqref{eq:deterministic-obj}. Define a \emph{causal policy} as one that, on each slot $t$, selects $\mathbf{p}[t] \in \mathcal{P}$ based only on information up to the start of slot $t$ (in particular, without knowledge
of $\omega[t]$).   Since $\omega[t]$ is i.i.d. over slots, any causal policy must have $\mathbf{p}[t]$ and $\omega[t]$ independent for all $t$. The next lemma shows that no causal policy $\mathbf{p}[t], t\in\{1,2,\ldots\}$ satisfying \eqref{eq:energy-availability}-\eqref{eq:energy-queue-dynamic} can attain a better utility than $U^\ast$.  

\begin{Lem} \label{lm:utility-upper-bound}
Let $\mathbf{p}[t]\in \mathcal{P}, t\in\{1,2,\ldots\}$ be yielded by any causal policy that 
consumes less energy than it harvests in the long term, so $\limsup_{T\rightarrow\infty} \frac{1}{T}\sum_{t=1}^{T} \expect{ \sum_{i=1}^n p_i[t]}  \leq \expect{e}$. Then, 
$$\limsup_{T\rightarrow \infty} \frac{1}{T} \sum_{t=1}^{T}\mathbb{E}[U(\mathbf{p}[t]; \omega[t])] \leq U^\ast.$$
where $U^\ast$ is the optimal utility value of problem \eqref{eq:deterministic-obj}-\eqref{eq:deterministic-box-cons}. 
\end{Lem}
\begin{proof} 
Fix a slot $t\in\{1,2,\ldots\}$.  Then 
\begin{equation} \label{eq:causal}
 \expect{U(\mathbf{p}[t]; \omega[t])} \overset{(a)}{=} \expect{\expect{U(\mathbf{p}[t];\omega[t])|\mathbf{p}[t]}} \overset{(b)}{=} \expect{h(\mathbf{p}[t])}
 \end{equation} 
where (a) holds by iterated expectations; (b) holds because $\mathbf{p}[t]$ and $\omega[t]$ are independent (by causality).

For each $T>0$ define $\bar{\mathbf{p}}[T] = [\bar{p}_1[T], \ldots, \bar{p}_n[T]]\tran$ with 
 $$ \bar{p}_i [T]= \frac{1}{T}\sum_{t=1}^T \expect{p_i[t]}, \forall i\in\{1,2,\ldots,n\}.$$  We know by assumption that: 
\begin{equation} \label{eq:by-assumption} 
\limsup_{T\rightarrow\infty}  \sum_{i=1}^n \bar{p}_i[T] \leq \expect{e} 
\end{equation} 
Further, since $\mathbf{p}[t]\in \mathcal{P}$ for all slots $t$,
it holds that $\bar{\mathbf{p}}[T] \in \mathcal{P}$ for all $T>0$. 
Also, 
\begin{align*}
\frac{1}{T}\sum_{t=1}^T\mathbb{E}[U(\mathbf{p}[t]; \omega[t])] &\overset{(a)}{=} \frac{1}{T}\sum_{t=1}^T \expect{h(\mathbf{p}[t])} \\
&\overset{(b)}{\leq} h\left(\expect{\mbox{$\frac{1}{T}\sum_{t=1}^T$} \mathbf{p}[t]} \right) \\
&= h(\bar{\mathbf{p}}[T])
\end{align*}
where (a) holds by \eqref{eq:causal}; (b) holds by Jensen's inequality for the concave function $h$.  It follows that: 
$$ \limsup_{T\rightarrow\infty} \frac{1}{T}\sum_{t=1}^T\mathbb{E}[U(\mathbf{p}[t]; \omega[t])] \leq \limsup_{T\rightarrow\infty} h(\bar{\mathbf{p}}[T]).$$
Define $\theta =  \limsup_{T\rightarrow\infty} h(\bar{\mathbf{p}}[T])$. It suffices to show that $\theta \leq U^{\ast}$. 
Since $\bar{\mathbf{p}}[T]$ is in the compact set $\mathcal{P}$ for all $T>0$, the Bolzano-Wierstrass theorem ensures there
is a subsequence of times $T_k$ such that $\bar{\mathbf{p}}[T_k]$ converges to a fixed vector $\mathbf{p}_0 \in \mathcal{P}$ and 
$h(\bar{\mathbf{p}}[T_k])$ converges to $\theta$ as $k\rightarrow \infty$: 
\begin{align*}
\lim_{k\rightarrow\infty} \bar{\mathbf{p}}[T_k] &= \mathbf{p}_0 \in \mathcal{P} \\
\lim_{k\rightarrow\infty} h(\bar{\mathbf{p}}[T_k]) &= \theta 
\end{align*}
Continuity of $h$ implies that $h(\mathbf{p}_0)=\theta$. 
By \eqref{eq:by-assumption} the vector $\mathbf{p}_0=[p_{0,1}, \ldots, p_{0,n}]\tran$ must satisfy $\sum_{i=1}^n p_{0,i} \leq \expect{e}$. Hence, $\mathbf{p}_0$ is a vector that 
satisfies constraints 
\eqref{eq:deterministic-energy-balance}-\eqref{eq:deterministic-box-cons} and achieves utility $h(\mathbf{p}_0) = \theta$. Since $U^\ast$ is defined as the optimal utility value to problem 
\eqref{eq:deterministic-obj}-\eqref{eq:deterministic-box-cons}, it holds that $\theta \leq U^{\ast}$. 
\end{proof} 
Note that the $U^\ast$ utility upper bound of Lemma \ref{lm:utility-upper-bound} holds for any policy that 
consumes no more energy than it harvests in the long term.  Policies that satisfy the physical battery constraints \eqref{eq:energy-availability}-\eqref{eq:energy-queue-dynamic} certainly consume no more energy than harvested in the long term.  However,  Lemma \ref{lm:utility-upper-bound} even holds for policies that violate these physical battery constraints.  For example, $U^\ast$ is still a valid bound for a policy that is allowed to ``borrow" energy from an external power source when its battery is empty and ``return" energy when its battery is full. Note that utility upper bounds were previously developed in \cite{Srivastava13TON, Arafa17ISIT} for energy harvesting problems with a fixed non-decreasing concave utility function. However, this paper considers power allocation for energy harvesting scenarios with time-varying stochastic concave utility functions.

\section{New Algorithm}

This subsection proposes a new learning aided dynamic power control algorithm that chooses power control actions based on system history, without requiring the current system state or its probability distribution. 
\subsection{New Algorithm}
\begin{algorithm}
\caption{New Algorithm}
\label{alg:new-alg}
Let $V>0$ be a constant algorithm parameter. Initialize virtual battery queue variable $Q[0] = 0$. Choose $\mathbf{p}[1] = [0, 0, \ldots, 0]\tran$ as the power action on slot $1$.  At the \emph{end} of each slot $t\in\{1,2,\ldots\}$, observe $\omega[t] = (e[t],\mathbf{s}[t])$ and do the following:
\begin{itemize}[leftmargin=10pt]
\item {\bf Update virtual battery queue $Q[t]$}: Update $Q[t]$ via:
\begin{align}
Q[t] = \min\{Q[t-1] + e[t] - \sum_{i=1}^n p_i[t],0\}. \label{eq:virtual-queue-Q-update}
\end{align} 
\item {\bf Power control}: Choose 
{\small
\begin{align}
\mathbf{p}[t+1] = \mbox{Proj}_{\mathcal{P}}\Big\{ \mathbf{p}[t]  + \frac{1}{V} \nabla_\mathbf{p} U(\mathbf{p}[t]; \omega[t])  +\frac{1}{V^2}Q[t] \mathbf{1}\Big\} \label{eq:power-update}
\end{align}}%
as the power action for the next slot $t+1$ where $\mbox{Proj}_{\mathcal{P}}\{\cdot\}$ represents the projection onto set $\mathcal{P}$, $\mathbf{1}$ denotes a column vector of all ones and $\nabla_\mathbf{p} U(\mathbf{p}[t]; \omega[t]) $ represents a subgradient (or gradient if $U(\mathbf{p}; \omega[t])$ is differentiable) vector of function $U(\mathbf{p}; \omega[t])$ at point $\mathbf{p} = \mathbf{p}[t]$.  Note that $\mathbf{p}[t]$, $Q[t]$ and $\nabla_\mathbf{p} U(\mathbf{p}[t]; \omega[t])$ are given constants in \eqref{eq:power-update}.
\end{itemize}
\end{algorithm}
The new dynamic power control algorithm is described in Algorithm \ref{alg:new-alg}. At the end of slot $t$, Algorithm \ref{alg:new-alg} chooses $\mathbf{p}[t+1]$  based on $\omega[t]$ without requiring $\omega[t+1]$.  To enable these decisions, the algorithm introduces a (nonpositive)  \emph{virtual} battery queue process $Q[t]\leq 0$, which shall later be shown to be related to a shifted version of the physical battery queue $E[t]$. (See e.g., equation \eqref{eq:E-Q-shift} in Theorem \ref{thm:energy-availability}.)

Note that Algorithm \ref{alg:new-alg} does not explicitly enforce the energy availability constraint \eqref{eq:energy-availability}. Let $\mathbf{p}[t+1]$ be given by \eqref{eq:power-update}, one may expect to use 
\begin{align}
\hat{\mathbf{p}}[t+1] = \frac{\min\{ \sum_{i=1}^n p_i[t+1], E[t]\}}{\sum_{i=1}^n p_i[t+1]} \mathbf{p}[t+1]\label{eq:explicit-energy-availability}
\end{align}
that scales down $\mathbf{p}[t+1]$ to enforce the energy availability constraint \eqref{eq:energy-availability}. However, our analysis in Section \ref{sec:performance-analysis} shows that if the battery capacity is at least as large as an $O(V)$ constant, then directly 
using $\mathbf{p}[t+1]$ from \eqref{eq:power-update} is ensured to always 
satisfy the energy availability constraint \eqref{eq:energy-availability}. Thus, there is no need to take the additional step \eqref{eq:explicit-energy-availability}.

\subsection{Algorithm Intuitions} \label{sec:alg-intuition}
\begin{Lem} \label{lm:power-control-minimization-interpret} The power control action $\mathbf{p}[t+1]$ chosen in \eqref{eq:power-update} is to solve the following quadratic convex program
\begin{align}
\max_{\mathbf{p}}~ & V(\nabla_\mathbf{p} U(\mathbf{p}[t]; \omega[t]))\tran (\mathbf{p} - \mathbf{p}[t]) +Q[t] \mathbf{1}\tran \mathbf{p} \nonumber \\ &- \frac{V^{2}}{2} \Vert \mathbf{p} - \mathbf{p}[t]\Vert^2   \label{eq:opt-power-control-obj}\\
\text{s.t.}\quad &  \mathbf{p} \in \mathcal{P} \label{eq:opt-power-control-cons}
\end{align}
\end{Lem}
\begin{proof}  By the definition of projection, equation \eqref{eq:power-update} is to solve $\min_{\mathbf{p}\in \mathcal{P}} \Vert \mathbf{p}- \big(\mathbf{p}[t] + \frac{1}{V} \nabla_\mathbf{p} U(\mathbf{p}[t]; \omega[t]) + \frac{1}{V^2}Q[t]\mathbf{1}\big) \Vert^2$. By expanding the square, eliminating constant terms and converting the minimization to a maximization of its negative object, it is easy to show this problem is equivalent to problem \eqref{eq:opt-power-control-obj}-\eqref{eq:opt-power-control-cons}.
\end{proof}

The convex projection \eqref{eq:power-update}, or equivalently, the quadratic convex program \eqref{eq:opt-power-control-obj}-\eqref{eq:opt-power-control-cons} can be easily solved. See e.g., Lemma 3 in \cite{YuNeely17INFOCOM} for an algorithm that solves an $n$-dimensional quadratic program over set $\mathcal{P}$ with complexity $O(n\log n)$. Thus, the overall complexity of  Algorithm \ref{alg:new-alg} is low.

\begin{enumerate}[leftmargin=15pt]
\item \emph{Connections with the drift-plus-penalty (DPP) technique for Lyapunov opportunistic optimization}: The Lyapunov opportunistic optimization solves stochastic optimization without distribution information by developing dynamic policies that adapt control actions to the current system state \cite{Tassiulas92TAC,Neely05JSAC,Eryilmaz06JSAC,Stolyar05,Neely08TON,book_Neely10}. The dynamic policy from Lyapunov opportunistic optimization can be interpreted as choosing control actions to maximize a DPP expression on each slot. Unfortunately,  the problem considered in this paper is different from the conventional Lyapunov opportunistic optimization problem since the power decision cannot be adapted to the unknown current system state. Nevertheless, if we treat $V (\nabla_\mathbf{p} U(\mathbf{p}[t]; \omega[t]))\tran (\mathbf{p} - \mathbf{p}[t]) -\frac{V^2}{2} \Vert \mathbf{p} - \mathbf{p}[t] \Vert^2$ as a penalty term and $Q[t]\mathbf{1}\tran \mathbf{p}$ as a drift term, then Lemma \ref{lm:power-control-minimization-interpret} suggests that the power control in Algorithm \ref{alg:new-alg} can still be interpreted as maximizing a (different) DPP expression. However, this DPP expression is significantly different from those conventional ones used in Lyapunov opportunistic optimization \cite{book_Neely10}.  Also, the penalty term $VU(\mathbf{p}[t+1]; \omega[t+1])$ used in conventional Lyapunov opportunistic optimization of \cite{book_Neely10} is unavailable in our problem since it depends on the unknown $\omega[t+1]$. 

\item \emph{Connections with online convex learning}: Online convex learning is a multi-round process where a decision maker selects its action from a \emph{fixed} set at each round before observing the corresponding utility function \cite{Zinkevich03ICML,book_PredictionLearningGames, Shalev-Shwartz11FoundationTrends}. If we assume the wireless device is equipped with an external free power source with infinite energy, i.e., the energy availability constraint \eqref{eq:energy-availability} is dropped, then the problem setup in this paper is similar to an online learning problem where the decision maker selects $\mathbf{p}[t+1] \in \mathcal{P}$ on each slot $t+1$ to maximize an unknown reward function $U(\mathbf{p}[t+1]; \omega[t+1])$ based on the information of previous reward functions $U(\mathbf{p}[\tau]; \omega[\tau]), \tau\in\{1,\ldots,t\}$. In this case, Zinkevich's online gradient method \cite{Zinkevich03ICML}, given by 
\begin{align}
\mathbf{p}[t+1] = \mbox{Proj}_{\mathcal{P}} \{ \mathbf{p}[t] + \gamma \nabla_\mathbf{p} U(\mathbf{p}[t]; \omega[t]) \} \label{eq:Zinkevich}
\end{align}
where $\gamma$ is a learning rate parameter, can solve this idealized problem.  In fact, if we ignore $\frac{1}{V^2}Q[t] \mathbf{1}$ involved in \eqref{eq:power-update}, then \eqref{eq:power-update} is identical to Zinkevich's learning algorithm with  $\gamma = 1/V$.  However, Zinkevich's algorithm and its variations \cite{Zinkevich03ICML,Hazan07ML, Shalev-Shwartz11FoundationTrends} require actions to be chosen from a \emph{fixed} set.  Our problem requires $\mathbf{p}[t]$ chosen on each slot $t$ to satisfy the energy availability constraint \eqref{eq:energy-availability}, which is time-varying since $E[t]$ evolves over time based on random energy arrivals and  previous power allocation decisions. 
\end{enumerate}

Now, it is clear why Algorithm \ref{alg:new-alg} is called a learning aided dynamic power control algorithm:  Algorithm \ref{alg:new-alg} can be viewed as an enhancement of the DPP technique originally developed for Lyapunov opportunistic optimization by replacing its penalty term with an expression used in Zinkevich's online gradient learning. 

\subsection{Main Results}
While the above subsection provides intuitive connections to prior work, note that existing techniques cannot be applied to our problem. The next section develops a novel performance analysis (summarized in Theorems \ref{thm:utility-performance} and \ref{thm:energy-availability}) to show that if $E[0] = E^{\max} = O(V)$, then the power control actions from Algorithm \ref{alg:new-alg} are ensured to satisfy the energy availability constraint \eqref{eq:energy-availability} and achieve
$$\frac{1}{t}\sum_{\tau=1}^{t} \mathbb{E}[U(\mathbf{p}[\tau]; \omega[\tau])] \geq U^{\ast} -  O(\frac{V}{t}) - O(\frac{1}{V}). $$
That is, for any desired $\epsilon>0$, by choosing $V = 1/\epsilon$ in Algorithm \ref{alg:new-alg}, 
we can attain an $O(\epsilon)$ optimal utility for all $ t\geq \Omega(\frac{1}{\epsilon^{2}})$ by using a battery with capacity $O(1/\epsilon)$.

\section{Performance Analysis of Algorithm \ref{alg:new-alg}} \label{sec:performance-analysis}
This section shows Algorithm \ref{alg:new-alg} can attain an $O(\epsilon)$ close-to-optimal utility by using a battery with capacity $O(1/\epsilon)$.  

\subsection{Drift Analysis}
Define $L[t] = \frac{1}{2} (Q[t])^2$ and call it a \emph{Lyapunov function}. Define the \emph{Lyapunov drift} as 
\begin{align*}
\Delta[t] = L[t+1] - L[t]
\end{align*}
\begin{Lem} \label{lm:drift}
Under Algorithm \ref{alg:new-alg}, for all $t\geq0$, the Lyapunov drift satisfies 
\begin{align}
\Delta[t] \leq Q[t] (e[t+1] - \sum_{i=1}^{n} p_{i}[t+1])   +\frac{1}{2}B
\end{align}
with constant $B = (\max\{e^{\max}, p^{\max}\})^2$, where $e^{\max}$ is the constant defined in Assumption \ref{ass:basic}.
\end{Lem}
\begin{proof}
Fix $t\geq 0$. Recall that for any $x\in \mathbb{R}$ if $y = \min\{x, 0\}$ then $y^2 \leq x^2$. It follows from  \eqref{eq:virtual-queue-Q-update} that 
\begin{align*}
(Q[t+1])^2 \leq (Q[t] + e[t+1] - \sum_{i=1}^{n}p_{i}[t+1])^2 .
\end{align*}
Expanding the square on the right side, dividing both sides by $2$ and rearranging terms yields $\Delta[t] \leq Q[t] (e[t+1] -\sum_{i=1}^{n}p_{i}[t+1])   + \frac{1}{2}(e[t+1] -\sum_{i=1}^{n}p_{i}[t+1])^2$.

This lemma follows by noting that $\vert e[t+1]  - \sum_{i=1}^{n}p_{i}[t+1]\vert \leq \max\{e^{\max}, p^{\max}\}$ since $0\leq \sum_{i=1}^{n}p_{i}[t+1]\leq p^{\max}$ and $0\leq e[t+1]\leq e^{\max}$.
\end{proof}


Recall that a function $f: \mathcal{Z}\mapsto \mathbb{R}$ is said to be \emph{strongly concave} with modulus $\alpha$ if there exists a constant $\alpha>0$ such that $f(\mathbf{z}) + \frac{1}{2} \alpha \Vert \mathbf{z} \Vert^2$ is concave on $\mathcal{Z}$. It is easy to show that if $f(\mathbf{z})$ is concave and $\alpha>0$, then $f(\mathbf{z}) - \frac{\alpha}{2} \Vert \mathbf{z} - \mathbf{z}_0\Vert^2$ is strongly concave with modulus $\alpha$ for any constant $\mathbf{z}_0$. The maximizer of a strongly concave function satisfies the following lemma: 

\begin{Lem} [Corollary 1 in \cite{YuNeely17SIOPT}]\label{lm:strong-convex-quadratic-optimality}
Let $\mathcal{Z} \subseteq \mathbb{R}^{n}$ be a convex set. Let function $f$ be strongly concave on $\mathcal{Z}$ with modulus $\alpha$ and $\mathbf{z}^{opt}$ be a global maximum of {\color{blue}$f$} on $\mathcal{Z}$. Then, $f(\mathbf{z}^{opt}) \geq f(\mathbf{z}) + \frac{\alpha}{2} \Vert \mathbf{z}^{opt} - \mathbf{z}\Vert^{2}$ for all $\mathbf{z}\in \mathcal{Z}$.
\end{Lem}
  
\begin{Lem}\label{lm:dpp-bound}
Let $U^{\ast}$ be the utility upper bound defined in Lemma \ref{lm:utility-upper-bound} and $\mathbf{p}^\ast $ be an optimal solution to problem \eqref{eq:deterministic-obj}-\eqref{eq:deterministic-box-cons} that attains $U^{\ast}$. At each iteration $t\in\{1,2,\ldots\}$, Algorithm \ref{alg:new-alg} guarantees 
\begin{align*}
V\mathbb{E}[U(\mathbf{p}[t]; \omega[t])] - \Delta[t] \geq V U^{\ast} +\frac{V^{2}}{2}\mathbb{E}[\Phi[t]] - \frac{D^{2} + B}{2}
\end{align*}
where $\Phi[t] =  \Vert \mathbf{p}^{\ast} - \mathbf{p}[t+1]\Vert^2- \Vert \mathbf{p}^{\ast} - \mathbf{p}[t]\Vert^2$, $D$ is the constant defined in Assumption \ref{ass:basic} and $B$ is the constant defined in Lemma \ref{lm:drift}.
\end{Lem}

\begin{proof}
Note that $\sum_{i=1}^{n} p_{i}^{\ast} \leq \mathbb{E}[e]$. Fix $t\in\{1,2,\ldots\}$. Note that $V(\nabla_\mathbf{p} U(\mathbf{p}[t]; \omega[t]))\tran(\mathbf{p} - \mathbf{p}[t]) + Q[t]\sum_{i=1}^np_i$ is a linear function with respect to $\mathbf{p}$. It follows that
{\small
\begin{align}
V \big(\nabla_\mathbf{p} U(\mathbf{p}[t]; \omega[t])\big)\tran (\mathbf{p} - \mathbf{p}[t]) + Q[t]\sum_{i=1}^{n} p_{i} - \frac{V^{2}}{2} \Vert \mathbf{p} - \mathbf{p}[t]\Vert^2 \label{eq:pf-lm-inequality-from-decision-eq1}
\end{align}}%
is strongly concave with respect to $\mathbf{p}\in \mathcal{P}$ with modulus $V^{2}$.  Since $\mathbf{p}[t+1]$ is chosen to maximize \eqref{eq:pf-lm-inequality-from-decision-eq1} over all $\mathbf{p}\in \mathcal{P}$,  and since $\mathbf{p}^{\ast} \in \mathcal{P}$, by Lemma \ref{lm:strong-convex-quadratic-optimality}, we have
\begin{align*}
&V \big(\nabla_\mathbf{p} U(\mathbf{p}[t]; \omega[t])\big)\tran (\mathbf{p}[t+1] - \mathbf{p}[t]) +Q[t]\sum_{i=1}^{n}p_{i}[t+1] \\&- \frac{V^{2}}{2} \Vert \mathbf{p}[t+1] - \mathbf{p}[t]\Vert^2 \\
\geq & V\big(\nabla_\mathbf{p} U(\mathbf{p}[t]; \omega[t])\big)\tran (\mathbf{p}^{\ast} - \mathbf{p}[t]) + Q[t]\sum_{i=1}^{n}p_{i}^{\ast} \nonumber \\&  - \frac{V^{2}}{2} \Vert \mathbf{p}^{\ast}- \mathbf{p}[t]\Vert^2 + \frac{V^{2}}{2} \Vert\mathbf{p}^{\ast} - \mathbf{p}[t+1]\Vert^2 \\
= &V\big(\nabla_\mathbf{p} U(\mathbf{p}[t]; \omega[t])\big)\tran (\mathbf{p}^{\ast} - \mathbf{p}[t]) + Q[t]\sum_{i=1}^{n}p_{i}^{\ast} + \frac{V^{2}}{2} \Phi[t]. 
\end{align*}
Subtracting $Q[t]e[t+1]$ from both sides and rearranging terms yields 
\begin{align*}
&V \big(\nabla_{\mathbf{p}}U(\mathbf{p}[t]; \omega[t])
\big)\tran(\mathbf{p}[t+1] - \mathbf{p}[t]) \\&+ Q[t](\sum_{i=1}^{n}p_{i}[t+1] -e[t+1]) \\
\geq&  V \big(\nabla_{\mathbf{p}}U(\mathbf{p}[t]; \omega[t]) \big)\tran(\mathbf{p}^{\ast} - \mathbf{p}[t]) + Q[t](\sum_{i=1}^{n}p_{i}^{\ast} -e[t+1])  \\&+\frac{V^{2}}{2} \Phi[t]+ \frac{V^{2}}{2} \Vert\mathbf{p}[t+1] - \mathbf{p}[t]\Vert^{2}.
\end{align*}
Adding $VU(\mathbf{p}[t]; \omega[t])$ to both sides and noting that $ U(\mathbf{p}[t]; \omega[t]) + (\nabla_\mathbf{p} U(\mathbf{p}[t]; \omega[t]))\tran (\mathbf{p}^{\ast} - \mathbf{p}[t]) \geq U(\mathbf{p}^{\ast}; \omega[t])$ by the concavity of $U(\mathbf{p}; \omega [t])$ yields
\begin{align*}
&VU(\mathbf{p}[t]; \omega[t]) + V \big(\nabla_\mathbf{p} U(\mathbf{p}[t]; \omega[t])\big)\tran (\mathbf{p}[t+1] - \mathbf{p}[t]) \\ &+ Q[t](\sum_{i=1}^{n}p_{i}[t+1]-e[t+1]) \\
\geq &  VU(\mathbf{p}^{\ast}; \omega[t]) + Q[t](\sum_{i=1}^{n}p_{i}^{\ast} -e[t+1])  +\frac{V^{2}}{2}\Phi[t]\\&+ \frac{V^{2}}{2} \Vert \mathbf{p}[t+1] - \mathbf{p}[t]\Vert^{2}. 
\end{align*}
Rearranging terms yields
\begin{align}
&VU(\mathbf{p}[t]; \omega[t]) + Q[t](\sum_{i=1}^{n}p_{i}[t+1] -e[t+1]) \nonumber\\
\geq &  VU(\mathbf{p}^{\ast}; \omega[t]) + Q[t](\sum_{i=1}^{n}p_i^{\ast} -e[t+1])+\frac{V^{2}}{2}\Phi[t]  \nonumber\\&+ \frac{V^{2}}{2} \Vert \mathbf{p}[t+1] - \mathbf{p}[t]\Vert^{2} \nonumber \\ &-V \big(\nabla_\mathbf{p} U(\mathbf{p}[t]; \omega[t])\big)\tran (\mathbf{p}[t+1] - \mathbf{p}[t])  \label{eq:pf-lm-dpp-eq1}
\end{align}
Note that 
\begin{align}
&V \big( \nabla_\mathbf{p} U(\mathbf{p}[t]; \omega[t])\big)\tran (\mathbf{p}[t+1] - \mathbf{p}[t]) \nonumber \\
\overset{(a)}{\leq}&  \frac{1}{2} \Vert \nabla_\mathbf{p} U(\mathbf{p}[t]; \omega[t]) \Vert^{2}  + \frac{V^{2}}{2} \Vert \mathbf{p}[t+1] - \mathbf{p}[t]\Vert^{2} \nonumber \\
\overset{(b)}{\leq}& \frac{1}{2}D^{2} + \frac{V^{2}}{2} \Vert \mathbf{p}[t+1] - \mathbf{p}[t]\Vert^{2} \label{eq:pf-lm-dpp-eq2}
\end{align}
where (a) follows by using basic inequality $\mathbf{x}\tran\mathbf{y} \leq \frac{1}{2}\Vert \mathbf{x}\Vert^{2} + \frac{1}{2}\Vert \mathbf{y}\Vert^{2}$ for all $\mathbf{x}, \mathbf{y}\in \mathbb{R}^{n}$ with $\mathbf{x} = \nabla_\mathbf{p} U(\mathbf{p}[t]; \omega[t])$ and $\mathbf{y} = V(\mathbf{p}[t+1] - \mathbf{p}[t])$; and (b) follows from Assumption \ref{ass:basic}.
Substituting \eqref{eq:pf-lm-dpp-eq2} into \eqref{eq:pf-lm-dpp-eq1} yields
\begin{align}
&VU(\mathbf{p}[t]; \omega[t]) + Q[t](\sum_{i=1}^{n}p_{i}[t+1] -e[t+1]) \nonumber\\
\geq &  VU(\mathbf{p}^{\ast}; \omega[t]) + Q[t](\sum_{i=1}^{n}p_{i}^{\ast}-e[t+1])+\frac{V^{2}}{2}\Phi[t]-\frac{1}{2}D^{2} \label{eq:pf-lm-dpp-eq3}
\end{align}
By Lemma \ref{lm:drift}, we have 
\begin{align}
-\Delta[t] \geq Q[t] (\sum_{i=1}^{n}p_{i}[t+1]  - e[t+1]) -\frac{B}{2}  \label{eq:pf-lm-dpp-eq4}
\end{align}
Summing \eqref{eq:pf-lm-dpp-eq3} and \eqref{eq:pf-lm-dpp-eq4}; and cancelling common terms on both sides yields
\begin{align}
&VU(\mathbf{p}[t]; \omega[t]) -\Delta[t] \nonumber \\
\geq&  VU(\mathbf{p}^{\ast}; \omega[t]) + Q[t](\sum_{i=1}^{n}p_{i}^{\ast}- e[t+1]) +\frac{V^{2}}{2} \Phi[t] \nonumber \\ &-\frac{D^{2} +B}{2} \label{eq:pf-lm-dpp-eq7}
\end{align}
Note that each $Q[t]$ (depending only on $e[\tau], p[\tau]$ with $\tau\in\{1,2,\ldots,t\}$) is independent of $e[t+1]$. Thus,
\begin{align}
&\mathbb{E}[Q[t](\sum_{i=1}^{n}p_{i}^{\ast}- e[t+1])] \nonumber\\
=& \mathbb{E}[Q[t]]\mathbb{E}[\sum_{i=1}^{n}p_{i}^{\ast} - e[t+1]] \nonumber\\
\overset{(a)}{\geq}& 0 \label{eq:pf-lm-dpp-eq8}
\end{align}
where (a) follows because $Q[t]\leq 0$ and $\sum_{i=1}^{n}p_{i}^{\ast} \leq \mathbb{E}[e]$ (recall that $e[t+1]$ is an i.i.d. sample of $e$). 

Taking expectations on both sides of \eqref{eq:pf-lm-dpp-eq7} and using \eqref{eq:pf-lm-dpp-eq8} and $\mathbb{E}[U(\mathbf{p}^{\ast}; \omega[t])] = U^{\ast}$ yields the desired result.
\end{proof}

\subsection{Utility Optimality Analysis}
The next theorem summarizes that the average expected utility attained by Algorithm \ref{alg:new-alg} is  within an $O(1/V)$ distance to $U^{\ast}$ defined in Lemma \ref{lm:utility-upper-bound}. 

\begin{Thm} \label{thm:utility-performance}
Let $U^{\ast}$ be the utility bound defined in Lemma \ref{lm:utility-upper-bound}. For all $t\in\{1,2,\ldots\}$, Algorithm \ref{alg:new-alg} guarantees
\begin{align}
\frac{1}{t}\sum_{\tau=1}^{t} \mathbb{E}[U(\mathbf{p}[\tau]; \omega[\tau])] \geq U^{\ast} - \frac{V(p^{\max})^{2}}{2t} -\frac{B}{2Vt} -  \frac{D^{2} + B}{2V} \label{eq:thm-utility-performance-eq1}
\end{align}
where  $D$ is the constant defined in Assumption \ref{ass:basic} and $B$ is the constant defined in Lemma \ref{lm:drift}. This implies, 
\begin{align}
\limsup_{t\rightarrow\infty}\frac{1}{t}\sum_{\tau=1}^{t} \mathbb{E}[U(\mathbf{p}[\tau]; \omega[\tau])] \geq U^{\ast}  -  \frac{D^{2} + B}{2V}.\label{eq:thm-utility-performance-eq2}
\end{align}
In particular, if we take $V = 1/\epsilon$ in Algorithm \ref{alg:new-alg}, then 
\begin{align}
\frac{1}{t} \sum_{\tau=1}^{t}\mathbb{E}[U(\mathbf{p}[\tau]; \omega[\tau])] \geq U^{\ast} -O(\epsilon), \forall t\geq \Omega(\frac{1}{\epsilon^{2}}).\label{eq:thm-utility-performance-eq3}
\end{align}
\end{Thm}

\begin{proof}
Fix $t\in\{1,2,\ldots\}$. For each $\tau \in\{1,2,\ldots, t\}$, by Lemma 
\ref{lm:dpp-bound}, we have $$\mathbb{E}[V U(\mathbf{p}[\tau]; \omega[\tau])] -\mathbb{E}[\Delta[\tau]] \geq V U^{\ast} +\frac{V^{2}}{2}\mathbb{E}[\Phi[\tau]] - \frac{D^{2} + B}{2}.$$
Summing over $\tau\in\{1,2,\ldots,t\}$, dividing both sides by $Vt$ and rearranging terms yields
\begin{align*}
&\frac{1}{t}\sum_{\tau=1}^{t} \mathbb{E}[U(\mathbf{p}[\tau]; \omega[\tau])] \\
\geq &U^{\ast}  + \frac{V}{2t} \sum_{\tau=1}^{t} \mathbb{E}[\Phi[\tau]] + \frac{1}{Vt} \sum_{\tau=1}^{t} \mathbb{E}[\Delta[\tau]] -  \frac{D^{2} + B}{2V}\\
\overset{(a)}{=}&U^{\ast} + \frac{V}{2t} \mathbb{E} [\Vert \mathbf{p}^{\ast} - \mathbf{p}[t+1]\Vert^2- \Vert \mathbf{p}^{\ast} - \mathbf{p}[1]\Vert^2] \\&
+  \frac{1}{2Vt} \mathbb{E}[(Q[t+1])^{2} - (Q[1])^{2}] -  \frac{D^{2} + B}{2V}\\
\geq&U^{\ast} - \frac{V}{2t} \mathbb{E} [\Vert \mathbf{p}^{\ast} - \mathbf{p}[1]\Vert^2] -\frac{1}{2Vt}\mathbb{E}[(Q[1])^{2}]-  \frac{D^{2} + B}{2V}\\
\overset{(b)}{\geq}&U^{\ast} - \frac{V(p^{\max})^{2}}{2t} -\frac{B}{2Vt} -  \frac{D^{2} + B}{2V}
\end{align*}
where (a) follows by recalling that $\Phi[\tau] =  \Vert \mathbf{p}^{\ast} - \mathbf{p}[\tau+1] \Vert ^2- \Vert \mathbf{p}^{\ast} - \mathbf{p}[\tau]\Vert^2$ and $\Delta[\tau] = \frac{1}{2}(Q[\tau+1])^{2} - \frac{1}{2}(Q[\tau])^{2}$; and (b) follows because $\Vert \mathbf{p}^{\ast} - \mathbf{p}[1]\Vert = \Vert \mathbf{p}^{\ast} \Vert  =\sqrt{\sum_{i=1}^n (p_i^\ast)^2} \leq \sum_{i=1}^n p_i^\ast \leq p^{\max}$ and $\vert Q[1]\vert = \vert Q[0] + e[1] - \sum_{i=1}^{n} p_{i}[1]\vert = \vert e[1] - \sum_{i=1}^{n} p_{i}[1]\vert \leq \max\{e^{\max}, p^{\max}\} = \sqrt{B}$ where $B$ is defined in Lemma \ref{lm:drift}. So far we have proven \eqref{eq:thm-utility-performance-eq1}. 

Equation  \eqref{eq:thm-utility-performance-eq2} follows directly by taking $\limsup$ on both sides of \eqref{eq:thm-utility-performance-eq1}. Equation \eqref{eq:thm-utility-performance-eq3} follows by substituting $V = \frac{1}{\epsilon}$ and $t = \frac{1}{\epsilon^{2}}$ into  \eqref{eq:thm-utility-performance-eq1}.
\end{proof}

\subsection{Lower Bound for Virtual Battery Queue $Q[t]$}
Note that $Q[t]\leq 0$ by \eqref{eq:virtual-queue-Q-update}. This subsection further shows that $Q[t]$ is bounded from below.  The projection $\mbox{Proj}_{\mathcal{P}}\{\cdot\}$ satisfies the following lemma:
\begin{Lem} \label{lm:proj-sol}
For any $\mathbf{p}[t]\in \mathcal{P}$ and vector $\mathbf{b} \leq \mathbf{0}$, where $\leq$ between two vectors means component-wisely less than or equal to,  $\tilde{\mathbf{p}} = \mbox{Proj}_{\mathcal{P}}\{\mathbf{p}[t] + \mathbf{b}\}$ is given by
\begin{align}
\tilde{p}_i = \max\{p_i[t] + b_i, 0\}, \forall i \in\{1,2,\ldots, n\} \label{eq:proj-lemma-sol}.
\end{align}
\end{Lem}
\begin{proof}
Recall that projection $ \mbox{Proj}_{\mathcal{P}}\{\mathbf{p}[t] + \mathbf{b}\}$ by definition is to solve 
\begin{align}
\min_{\mathbf{p}}~ &  \sum_{i=1}^n (p_i - (p_i[t] + b_i))^2 \label{eq:pf-lm-proj-sol-eq1} \\
\text{s.t.}\quad &   \sum_{i=1}^n p_i \leq p^{\max} \label{eq:pf-lm-proj-sol-eq2}\\
& p_i \geq 0, \forall i\in\{1,2,\ldots, n\} \label{eq:pf-lm-proj-sol-eq3}
\end{align}
Let $\mathcal{I} \subseteq\{1,2,\ldots,n\}$ be the coordinate index set given by $\mathcal{I} = \{i\in \{1,2,\ldots,n\}: p_i[t] + b_j <0\}$.  For any $\mathbf{p}$ such that $\sum_{i=1}^n p_i \leq p^{\max}$ and $p_i \geq 0, \forall i\in\{1,2,\ldots, n\}$, we have 
\begin{align*}
&\sum_{i=1}^n (p_i - (p_i[t] + b_i))^2 \\
=& \sum_{i\in \mathcal{I}} (p_i - (p_i[t] + b_i))^2 + \sum_{i\in \{1,2,\ldots, n\}\setminus \mathcal{I}} (p_i - (p_i[t] + b_i))^2\\
\geq &\sum_{i\in \mathcal{I}} (p_i - (p_i[t] + b_i))^2\\
\overset{(a)}{\geq}&   \sum_{i\in \mathcal{I}} (p_i[t] + b_i)^2
\end{align*}
where (a) follows because $p_i[t] + b_i < 0$ for $i\in \mathcal{I}$ and $p_i \geq 0, \forall i\in\{1,2,\ldots,n\}$. Thus, $\sum_{i\in \mathcal{I}} (p_i[t] + b_i)^2$ is an object value lower bound of problem \eqref{eq:pf-lm-proj-sol-eq1}-\eqref{eq:pf-lm-proj-sol-eq3}.

Note that $\tilde{\mathbf{p}}$ given by \eqref{eq:proj-lemma-sol} is feasible to problem \eqref{eq:pf-lm-proj-sol-eq1}-\eqref{eq:pf-lm-proj-sol-eq3} since $\tilde{p}_i \geq 0, \forall i\in\{1,2,\ldots, n\}$ and $\sum_{i=1}^n \tilde{p}_i \leq \sum_{i=1}^n p_i[t] \leq p^{\max}$ because 
$\tilde{p}_i \leq p_i[t]$ for all $i$  and   $\mathbf{p}[t]\in \mathcal{P}$.  We further note that 
$$\sum_{i=1}^n (\tilde{p_i}- (p_i[t] + b_i))^2 = \sum_{i\in \mathcal{I}} (p_i[t] + b_i)^2.$$ 
That is, $\tilde{\mathbf{p}}$ given by \eqref{eq:proj-lemma-sol} attains the object value lower bound of problem \eqref{eq:pf-lm-proj-sol-eq1}-\eqref{eq:pf-lm-proj-sol-eq3} and hence is the optimal solution to problem \eqref{eq:pf-lm-proj-sol-eq1}-\eqref{eq:pf-lm-proj-sol-eq3}. Thus,  $\tilde{\mathbf{p}} = \mbox{Proj}_{\mathcal{P}}\{\mathbf{p}[t] + \mathbf{b}\}$.
\end{proof}

\begin{Cor}\label{cor:large-Q-decrease-p}
If $Q[t] \leq - V(D^{\max} +p^{\max})$ with $D^{\max} = \max\{D_1, \ldots, D_n\}$, then Algorithm \ref{alg:new-alg} guarantees
\begin{align*}
p_i[t+1] \leq \max \{p_i[t] - \frac{1}{V} p^{\max}, 0\}, \forall i\in\{1,2,\ldots,n\}.
\end{align*}
where $D_1, \ldots, D_n$ are constants defined in Assumption \ref{ass:basic}.
\end{Cor}
\begin{proof}
Let $\mathbf{b} = \frac{1}{V} \nabla_\mathbf{p} U(\mathbf{p}[t]; \omega[t])  +\frac{1}{V^2}Q[t]\mathbf{1}$. Since $\frac{\partial}{\partial p_i} U(\mathbf{p}[t]; \omega[t])  \leq D_i, \forall i\in\{1,2,\ldots, n\}$ by Assumption \ref{ass:basic} and $Q[t] \leq - V(D^{\max} +p^{\max})$, we know $b_i \leq -\frac{1}{V} p^{\max}, \forall i\in\{1,2,\ldots, n\}$. By Lemma \ref{lm:proj-sol}, we have
\begin{align*}
p_i[t+1] =& \max\{p_i[t] +b_i, 0\}\\
\leq & \max\{p_i[t] - \frac{1}{V} p^{\max}, 0\}, \forall i\in\{1,2,\ldots, n\}.
\end{align*}
\end{proof}
By Corollary \ref{cor:large-Q-decrease-p}, if $Q[t]\leq - V(D^{\max} +p^{\max})$, then each component of $\mathbf{p}[t+1]$ decreases by $\frac{1}{V}p^{\max}$ until it hits $0$. That is, if $Q[t]\leq - V(D^{\max} +p^{\max})$ for sufficiently many slots, Algorithm \ref{alg:new-alg} eventually chooses $\mathbf{0}$ as the power decision.  By virtual queue update equation \eqref{eq:virtual-queue-Q-update}, $Q[t]$ decreases only when $\sum_{i=1}^{n} p_i[t]>0$.  These two observations suggest that $Q[t]$ yielded by Algorithm \ref{alg:new-alg} should be eventually bounded from below. This is formally summarized in the next theorem. 

\begin{Thm}\label{thm:Q-bound} Define positive constant $Q^{l}$, where superscript l denotes ``lower" bound, as
\begin{align}
Q^{l} =& \lceil V\rceil(D^{\max} + 2p^{\max} + e^{\max})\label{eq:Q-lower-bound}
\end{align}
where $V>0$ is the algorithm parameter, $\lceil V \rceil$ represents the smallest integer no less than $V$, $e^{\max}$ is the constant defined in Assumption \ref{ass:basic} and $D^{\max}$ is the constant defined in Corollary \ref{cor:large-Q-decrease-p}.  Algorithm \ref{alg:new-alg} guarantees 
\begin{align*}
Q[t] \geq -Q^{l}, \forall t\in\{0,1,2,\ldots\}.
\end{align*}
\end{Thm}
\begin{proof}
By virtual queue update equation \eqref{eq:virtual-queue-Q-update}, we know $Q[t]$ can increase by at most $e^{\max}$ and can decrease by at most $p^{\max}$ on each slot. Since $Q[0] = 0$, we know $Q[t] \geq -Q^{l}$ for all $t\leq \lceil V \rceil$. We need to show $Q[t] \geq -Q^{l}$ for all $t>\lceil V \rceil$. This can be proven by contradiction as follows:

Assume $Q[t] < -Q^{l}$ for some $t>\lceil V \rceil$. Let $\tau > \lceil V \rceil$ be the \emph{first} (smallest) slot index when this happens. By the definition of $\tau$, we have  $Q[\tau] < -Q^{l}$ and 
\begin{align}
Q[\tau] < Q[\tau-1]. \label{eq:pf-Q-bound-eq1}
\end{align}
Now consider the value of $Q[\tau - \lceil V \rceil]$ in two cases (note that $\tau - \lceil V \rceil >0$). 
\begin{itemize}[leftmargin=10pt]
\item Case $Q[\tau-\lceil V \rceil] \geq -\lceil V \rceil(D^{\max}+p^{\max} + e^{\max})$: Since $Q[t]$ can decrease by at most $p^{\max}$ on each slot, we know $Q[\tau]\geq -\lceil V \rceil(D^{\max}+2p^{\max}+ e^{\max}) = -Q^{l}$. This contradicts the definition of $\tau$.
\item Case $Q[\tau-\lceil V \rceil] < -\lceil V \rceil(D^{\max}+p^{\max}+e^{\max})$: Since $Q[t]$ can increase by at most $e^{\max}$ on each slot, we know $Q[t] < -\lceil V \rceil(D^{\max}+p^{\max})$ for all $\tau-\lceil V \rceil\leq t\leq \tau-1$. By Corollary \ref{cor:large-Q-decrease-p}, for all $\tau-\lceil V \rceil\leq t\leq \tau-1$, we have 
\begin{align*}
p_i[t+1] \leq \max \{p_i[t] - \frac{1}{V} p^{\max}, 0\}, \forall i\in\{1,2,\ldots,n\}.
\end{align*}
Since the above inequality holds for all 
$ t\in\{\tau-\lceil V \rceil, \tau-\lceil V \rceil+1, \ldots,\tau-1\}$, 
and since at the start of this interval we trivially have $p_{i}[\tau-\lceil V \rceil] \leq p^{\max}, \forall i\in\{1,2,\ldots,n\}$, 
at each step of this interval each component of the power vector either hits zero or decreases by $\frac{1}{V} p^{\max}$, and so 
 after the $\lceil V \rceil$ steps of this interval 
 we have $p_{i}[\tau] = 0, \forall i\in\{1,2,\ldots,n\}$.  By  \eqref{eq:virtual-queue-Q-update}, we have 
\begin{align*}
Q[\tau] =& \min\{Q[\tau-1] + e[\tau] - \sum_{i=1}^{n} p_{i}[\tau], 0\} \\
=& \min\{Q[\tau-1] + e[\tau], 0\} \\
\geq& \min\{Q[\tau-1], 0\} \\
=& Q[\tau-1]
\end{align*}
 where the final equality holds because the queue is never positive (see \eqref{eq:virtual-queue-Q-update}). 
 This contradicts  \eqref{eq:pf-Q-bound-eq1}.
\end{itemize}
Both cases lead to contradictions. Thus, $Q[t] \geq -Q^{l}$ for all $t>\lceil V \rceil$.
\end{proof}

\subsection{Energy Availability Guarantee} \label{sec:energy-availability}
To implement the power decisions of Algorithm \ref{alg:new-alg} for the physical battery system $E[t]$ from equations
\eqref{eq:energy-availability}-\eqref{eq:energy-queue-dynamic}, 
we must ensure the energy availability constraint \eqref{eq:energy-availability} holds on each slot.  The next theorem shows that Algorithm \ref{alg:new-alg} ensures the constraint \eqref{eq:energy-availability} always holds as long as the battery capacity satisfies $E^{\max} \geq Q^l + p^{\max}$ and the initial energy satisfies $E[0] = E^{\max}$. It also explains that $Q[t]$ used in Algorithm \ref{alg:new-alg} is a shifted version of the physical battery backlog $E[t]$.

\begin{Thm} \label{thm:energy-availability}
If $E[0] = E^{\max} \geq Q^{l} + p^{\max} $, where $Q^{l}$ is the constant defined in Theorem \ref{thm:Q-bound}, then Algorithm \ref{alg:new-alg} ensures the energy availability constraint \eqref{eq:energy-availability} on each slot $t\in\{1,2,\ldots\}$. Moreover 
\begin{align}
E[t] = Q[t] + E^{\max}, \forall t\in\{0,1,2,\ldots\}. \label{eq:E-Q-shift}
\end{align}
\end{Thm}
\begin{proof}
Note that to show the energy availability constraint $\sum_{i=1}^n p_i[t]\leq E[t-1], \forall t\in\{1,2,\ldots\}$ is equivalent to show
\begin{align}
\sum_{i=1}^n p_i[t+1]\leq E[t], \forall t\in\{0, 1,2,\ldots\}. \label{eq:pf-thm-energy-availabity-eq1}
\end{align}
This lemma can be proven by inductions. 

Note that $E[0] = E^{\max}$ and  $Q[0] = 0$. It is immediate that \eqref{eq:E-Q-shift} holds for $t=0$.  Since $E[0] = E^{\max} \geq p^{\max}$ and $\sum_{i=1}^n p_i[1] \leq p^{\max}$, equation \eqref{eq:pf-thm-energy-availabity-eq1} also holds for $t=0$. Assume  \eqref{eq:pf-thm-energy-availabity-eq1} and \eqref{eq:E-Q-shift} hold for $t = t_0$ and consider $t = t_0+1$. By virtual queue dynamic \eqref{eq:virtual-queue-Q-update}, we have
\begin{align*}
Q[t_0+1] = \min\{Q[t_0]  + e[t_0] - \sum_{i=1}^n p_i[t_0], 0\}
\end{align*}
Adding $E^{\max}$ on both sides yields
\begin{align*}
&Q[t_0+1] + E^{\max} \\
=& \min\{Q[t_0]  + e[t_0+1] - \sum_{i=1}^n p_i[t_0+1] + E^{\max}, E^{\max}\}\\
\overset{(a)}{=}& \min\{E[t_0] + e[t_0+1] - \sum_{i=1}^n p_i[t_0+1], E^{\max}\}\\
\overset{(b)}{=} & E[t_0+1] 
\end{align*}
where (a) follows from the induction hypothesis $E[t_0] = Q[t_0] + E^{\max}$ and (b) follows from the energy queue dynamic \eqref{eq:energy-queue-dynamic}. Thus, \eqref{eq:E-Q-shift} holds for $t = t_0+1$.

Now observe
\begin{align*}
 E[t_0+1] &= Q[t_0+1] + E^{\max} \\
 &\overset{(a)}{\geq} E^{\max} - Q^{l} \\
 &\geq p^{\max} \\
 &\overset{(b)}{\geq} \sum_{i=1}^np_i[t_0+2]
 \end{align*}
 where (a) follows from the fact that $ Q[t] \geq -Q^{l}, \forall t\in\{0,1,2,\ldots\}$ by Theorem \ref{thm:Q-bound}; (b) holds since sum power is  never more than $p^{\max}$. Thus, \eqref{eq:pf-thm-energy-availabity-eq1} holds for $t=t_0+1$.

Thus, this theorem follows by induction.
\end{proof}

\subsection{Utility Optimality and Battery Capacity Tradeoff}

By Theorem \ref{thm:utility-performance}, Algorithm \ref{alg:new-alg} is guaranteed to attain a utility within an $O(1/V)$ distance to the optimal utility $U^{\ast}$. To obtain an $O(\epsilon)$-optimal utility, we can choose $V = 1/\epsilon $.  In this case, $Q^{l}$ defined in \eqref{eq:Q-lower-bound} is order $O(V)$. By Theorem \ref{thm:energy-availability},we need the battery capacity $E^{\max} \geq Q^{l} + p^{\max} = O(V) = O(1/\epsilon)$ to satisfy the energy availability constraint. Thus, there is a $[O(\epsilon), O(1/\epsilon)]$ tradeoff between the utility optimality and the required battery capacity.   On the other hand, if the battery capacity $E^{{\max}}$ is fixed and parameters $D^{\max}, e^{\max}$ in Assumption \ref{ass:basic} can be accurately estimated, our Theorems \ref{thm:Q-bound} and \ref{thm:energy-availability} together imply that Algorithm 1 ensures energy availability constraint \eqref{eq:energy-availability} by choosing $\lceil V\rceil < (E^{\max} - p^{\max})/(D^{\max} + 2p^{\max} + e^{\max})$.

\subsection{Extensions} \label{sec:large-delay}
Thus far, we have assumed that $\omega[t]$ is known with one slot delay, i.e., at the end of slot $t$, or equivalently, at the beginning of slot $t+1$. In fact, if $\omega(t)$ is observed with $t_0$ slot delay (at the end of slot $t+t_0-1$), we can modify Algorithm \ref{alg:new-alg} by initializing $\mathbf{p}[\tau]=\mathbf{0},\tau \in \{1,2,\ldots, t_0\}$ and updating $Q[t-t_0+1] = \min\{Q[t-t_0] + e[t-t_0+1]  -\sum_{i=1}^{n}p_{i}[t-t_0+1], 0\}$, $\mathbf{p}[t+1] = \mbox{Proj}_{\mathcal{P}}\{ \mathbf{p}[t-t_0+1]  + \frac{1}{V} \nabla_\mathbf{p} U(\mathbf{p}[t-t_0+1]; \omega[t-t_0+1])  +\frac{1}{V^2}Q[t-t_0+1] \mathbf{1}\}$ at the end of each slot $t\in\{t_0, t_0+1, \ldots\}$.  By extending the analysis in this section (from a $t_0=1$ version to a general $t_0$ version), a similar $[O(\epsilon), O(1/\epsilon)]$ tradeoff can be established.

\section{Performance in Non I.I.D Systems} \label{section:non-iid} 

Thus far, we have assumed the system state $\{\omega[t]\}_{t=1}^{\infty}$ evolves in an i.i.d. manner.  We now address the issue of non-i.i.d. behavior.  Unfortunately, a counter-example in \cite{Mannor09JMLR} shows that, even for a simpler scenario of constrained online convex optimization with one arbitrarily time-varying objective and one arbitrarily time-varying constraint, it is impossible for any algorithm to achieve regret-based guarantees similar to those of the unconstrained case.  Intuitively, this is because decisions that lead to good behavior for the objective function over one time interval may incur undesirable performance for the constraint function over a larger time interval.  However, below we show that significant analytical guarantees can still be achieved by allowing the $s[t]$ process to be an arbitrary and non-i.i.d. process, while maintaining the structured independence assumptions for the $e[t]$ process.\footnote{In fact, our assumptions on $e[t]$ in this section are slightly more general than the i.i.d. assumption of previous sections.} In this section, we consider a more general system model described as follows:

\begin{Model} [{\bf Non i.i.d. System State}] 
Consider a stochastic system state process  $\{\omega[t]\}_{t=1}^{\infty}$, where $\omega[t]=(e[t], \mathbf{s}[t]) \in \Omega$ for all $t$, 
satisfying the following conditions:
\begin{enumerate}
\item $\{\mathbf{s}[t]\}_{t=1}^{\infty}$ is an arbitrary process, possibly time-correlated and with different distributions at each $t$. 

\item $\{e[t]\}_{t=1}^{\infty}$ is a sequence of independent variables that take values over $[0, e^{\max}]$ and that have
the same expectation $\mathbb{E}[e[t]] = \bar{e}$ at each $t$. For each $t \in \{1, 2, 3, ...\}$, the value of $e[t+1]$ is independent of $\{(e[\tau], \mathbf{s}[t])\}_{\tau=1}^{t}$. 
\end{enumerate}
\end{Model}

Note that the above stochastic model includes the i.i.d. $\omega[t]$ model considered in the previous sections as a special case. Under this generalized stochastic system model, we compare the performance of Algorithm \ref{alg:new-alg} against any fixed power action vector $\mathbf{q}$ satisfying $\sum_{i=1}^{n} q_{i} \leq \bar{e}$.  Specifically, we will show that for any $\mathbf{q}\in \mathcal{P}$ with $\sum_{i=1}^{n} q_{i} \leq \bar{e}$, Algorithm \ref{alg:new-alg} with $V = \frac{1}{\epsilon}$ and $E[0] = E^{\max} = O(\frac{1}{\epsilon})$ ensures
\begin{align}
\frac{1}{t} \sum_{\tau=1}^{t} \mathbb{E}[U(\mathbf{p}[\tau]; \omega[\tau])] \geq \frac{1}{t} \sum_{\tau=1}^{t} \mathbb{E}[U(\mathbf{q}; \omega[\tau])] - O(\epsilon), \forall t \geq \frac{1}{\epsilon^{2}} \label{eq:online-ulility-perf}
\end{align}
and 
$\mathbf{p}[t]$ satisfies the energy availability constraint \eqref{eq:energy-availability} on each slot $t$. Note that if we fix a positive integer $t$ and choose $\mathbf{q} = \argmin_{\mathbf{p}\in \mathcal{P}: \sum_{i=1}^{n} p_{i} \leq \bar{e}} \sum_{\tau=1}^{t} U(\mathbf{p}; \omega[\tau])$, which is the best fixed decision of $t$ slots in hindsight, then \eqref{eq:online-ulility-perf} implies Algorithm \ref{alg:new-alg} has $O(\sqrt{T})$ regret in the terminology of online learning \cite{Zinkevich03ICML, YuNeely17NIPS}.

In fact, it is easy to verify that all lemmas and theorems except Lemma \ref{lm:dpp-bound} and Theorem \ref{thm:utility-performance} in Section \ref{sec:performance-analysis} are sample path arguments that 
hold for arbitrary $\omega[t]$ processes (even for those violating the ``Non i.i.d. System State'' model defined above). It follows that $\mathbf{p}[t]$ from Algorithm \ref{alg:new-alg} satisfies the energy availability constraint \eqref{eq:energy-availability} on all slots for arbitrary process $\omega[t]$. To prove \eqref{eq:online-ulility-perf}, we generalize Lemma \ref{lm:dpp-bound} under the generalized ``Non i.i.d. System State'' model.

\begin{Lem}\label{lm:non-iid-dpp-bound}
Let $\mathbf{q}\in \mathcal{P}$ be any fixed vector satisfying $\sum_{i=1}^{n} q_{i} \leq \bar{e}$. Under the Non i.i.d. System State model, at each iteration $t\in\{1,2,\ldots\}$, Algorithm \ref{alg:new-alg} guarantees 
\begin{align*}
&V\mathbb{E}[U(\mathbf{p}[t]; \omega[t])] - \Delta[t] \\ 
\geq &V\mathbb{E}[U(\mathbf{q}; \omega[t])]  +\frac{V^{2}}{2}\mathbb{E}[\Psi[t]] - \frac{D^{2} + B}{2}
\end{align*}
where $\Psi[t] =  \Vert \mathbf{q} - \mathbf{p}[t+1]\Vert^2- \Vert \mathbf{q} - \mathbf{p}[t]\Vert^2$, $D$ is the constant defined in Assumption \ref{ass:basic} and $B$ is the constant defined in Lemma \ref{lm:drift}.
\end{Lem}
\begin{proof}
The proof is almost identical to the proof of Lemma \ref{lm:dpp-bound} until \eqref{eq:pf-lm-dpp-eq7}.  Fix $t\in\{1,2,\ldots\}$. As observed in the proof of Lemma \ref{lm:dpp-bound}, the expression \eqref{eq:pf-lm-inequality-from-decision-eq1} is strongly concave with respect to $\mathbf{p}\in \mathcal{P}$ with modulus $V^{2}$.  Since $\mathbf{q} \in \mathcal{P}$ and $\mathbf{p}[t+1]$ is chosen to maximize \eqref{eq:pf-lm-inequality-from-decision-eq1} over all $\mathbf{p}\in \mathcal{P}$, by Lemma \ref{lm:strong-convex-quadratic-optimality}, we have
\begin{align*}
&V \big(\nabla_\mathbf{p} U(\mathbf{p}[t]; \omega[t])\big)\tran (\mathbf{p}[t+1] - \mathbf{p}[t]) +Q[t]\sum_{i=1}^{n}p_{i}[t+1] \\&- \frac{V^{2}}{2} \Vert \mathbf{p}[t+1] - \mathbf{p}[t]\Vert^2 \\
\geq & V\big(\nabla_\mathbf{p} U(\mathbf{p}[t]; \omega[t])\big)\tran (\mathbf{q} - \mathbf{p}[t]) + Q[t]\sum_{i=1}^{n}q_{i} \nonumber \\&  - \frac{V^{2}}{2} \Vert \mathbf{q}- \mathbf{p}[t]\Vert^2 + \frac{V^{2}}{2} \Vert\mathbf{q} - \mathbf{p}[t+1]\Vert^2 \\
= &V\big(\nabla_\mathbf{p} U(\mathbf{p}[t]; \omega[t])\big)\tran (\mathbf{q} - \mathbf{p}[t]) + Q[t]\sum_{i=1}^{n}q_{i} + \frac{V^{2}}{2} \Psi[t]. 
\end{align*}
Subtracting $Q[t]e[t+1]$ from both sides and rearranging terms yields 
\begin{align*}
&V \big(\nabla_{\mathbf{p}}U(\mathbf{p}[t]; \omega[t])
\big)\tran(\mathbf{p}[t+1] - \mathbf{p}[t]) \\&+ Q[t](\sum_{i=1}^{n}p_{i}[t+1] -e[t+1]) \\
\geq&  V \big(\nabla_{\mathbf{p}}U(\mathbf{p}[t]; \omega[t]) \big)\tran(\mathbf{q} - \mathbf{p}[t]) + Q[t](\sum_{i=1}^{n}q_{i} -e[t+1])  \\&+\frac{V^{2}}{2} \Psi[t]+ \frac{V^{2}}{2} \Vert\mathbf{p}[t+1] - \mathbf{p}[t]\Vert^{2}.
\end{align*}
Adding $VU(\mathbf{p}[t]; \omega[t])$ to both sides and noting that $ U(\mathbf{p}[t]; \omega[t]) + (\nabla_\mathbf{p} U(\mathbf{p}[t]; \omega[t]))\tran (\mathbf{q} - \mathbf{p}[t]) \geq U(\mathbf{q}; \omega[t])$ by the concavity of $U(\mathbf{p}; \omega [t])$ yields
\begin{align*}
&VU(\mathbf{p}[t]; \omega[t]) + V \big(\nabla_\mathbf{p} U(\mathbf{p}[t]; \omega[t])\big)\tran (\mathbf{p}[t+1] - \mathbf{p}[t]) \\ &+ Q[t](\sum_{i=1}^{n}p_{i}[t+1]-e[t+1]) \\
\geq &  VU(\mathbf{q}; \omega[t]) + Q[t](\sum_{i=1}^{n}q_{i} -e[t+1])  +\frac{V^{2}}{2}\Psi[t]\\&+ \frac{V^{2}}{2} \Vert \mathbf{p}[t+1] - \mathbf{p}[t]\Vert^{2}. 
\end{align*}
Rearranging terms yields
\begin{align}
&VU(\mathbf{p}[t]; \omega[t]) + Q[t](\sum_{i=1}^{n}p_{i}[t+1] -e[t+1]) \nonumber\\
\geq &  VU(\mathbf{q}; \omega[t]) + Q[t](\sum_{i=1}^{n}q_i -e[t+1])+\frac{V^{2}}{2}\Phi[t]  \nonumber\\&+ \frac{V^{2}}{2} \Vert \mathbf{p}[t+1] - \mathbf{p}[t]\Vert^{2} \nonumber \\ &-V \big(\nabla_\mathbf{p} U(\mathbf{p}[t]; \omega[t])\big)\tran (\mathbf{p}[t+1] - \mathbf{p}[t])  \label{eq:pf-lm-non-iid-dpp-eq1}
\end{align}
Note that 
\begin{align}
&V \big( \nabla_\mathbf{p} U(\mathbf{p}[t]; \omega[t])\big)\tran (\mathbf{p}[t+1] - \mathbf{p}[t]) \nonumber \\
\overset{(a)}{\leq}&  \frac{1}{2} \Vert \nabla_\mathbf{p} U(\mathbf{p}[t]; \omega[t]) \Vert^{2}  + \frac{V^{2}}{2} \Vert \mathbf{p}[t+1] - \mathbf{p}[t]\Vert^{2} \nonumber \\
\overset{(b)}{\leq}& \frac{1}{2}D^{2} + \frac{V^{2}}{2} \Vert \mathbf{p}[t+1] - \mathbf{p}[t]\Vert^{2} \label{eq:pf-lm-non-iid-dpp-eq2}
\end{align}
where (a) follows by using basic inequality $\mathbf{x}\tran\mathbf{y} \leq \frac{1}{2}\Vert \mathbf{x}\Vert^{2} + \frac{1}{2}\Vert \mathbf{y}\Vert^{2}$ for all $\mathbf{x}, \mathbf{y}\in \mathbb{R}^{n}$ with $\mathbf{x} = \nabla_\mathbf{p} U(\mathbf{p}[t]; \omega[t])$ and $\mathbf{y} = V(\mathbf{p}[t+1] - \mathbf{p}[t])$; and (b) follows from Assumption \ref{ass:basic}.
Substituting \eqref{eq:pf-lm-non-iid-dpp-eq2} into \eqref{eq:pf-lm-non-iid-dpp-eq1} yields
\begin{align}
&VU(\mathbf{p}[t]; \omega[t]) + Q[t](\sum_{i=1}^{n}p_{i}[t+1] -e[t+1]) \nonumber\\
\geq &  VU(\mathbf{q}; \omega[t]) + Q[t](\sum_{i=1}^{n}q_{i}-e[t+1])+\frac{V^{2}}{2}\Psi[t]-\frac{1}{2}D^{2} \label{eq:pf-lm-non-iid-dpp-eq3}
\end{align}
By Lemma \ref{lm:drift}, we have 
\begin{align}
-\Delta[t] \geq Q[t] (\sum_{i=1}^{n}p_{i}[t+1]  - e[t+1]) -\frac{B}{2}  \label{eq:pf-lm-non-iid-dpp-eq4}
\end{align}
Summing \eqref{eq:pf-lm-non-iid-dpp-eq3} and \eqref{eq:pf-lm-non-iid-dpp-eq4}; and cancelling common terms on both sides yields
\begin{align}
&VU(\mathbf{p}[t]; \omega[t]) -\Delta[t] \nonumber \\
\geq&  VU(\mathbf{q}; \omega[t]) + Q[t](\sum_{i=1}^{n}q_{i}- e[t+1]) +\frac{V^{2}}{2} \Psi[t] \nonumber \\ &-\frac{D^{2} +B}{2} \label{eq:pf-lm-non-iid-dpp-eq7}
\end{align}
Note that each $Q[t]$ (depending only on $e[\tau], p[\tau]$ with $\tau\in\{1,2,\ldots,t\}$) is independent of $e[t+1]$ by our ``Non I.I.D System State'' model. Thus,
\begin{align}
&\mathbb{E}\left[Q[t](\sum_{i=1}^{n}q_{i}- e[t+1])\right] \nonumber\\
=& \mathbb{E}[Q[t]]\mathbb{E}[\sum_{i=1}^{n}q_{i} - e[t+1]] \nonumber\\
\overset{(a)}{\geq}& 0 \label{eq:pf-lm-non-iid-dpp-eq8}
\end{align}
where (a) follows because $Q[t]\leq 0$ and $\sum_{i=1}^{n}q_{i} \leq \bar{e} = \mathbb{E}[e[t+1]], \forall t$ where the last step follows from our ``Non I.I.D System State'' model. 

Taking expectations on both sides of \eqref{eq:pf-lm-non-iid-dpp-eq7} and using \eqref{eq:pf-lm-non-iid-dpp-eq8}  yields the desired result.
\end{proof}

Now \eqref{eq:online-ulility-perf} follows directly from Lemma \ref{lm:non-iid-dpp-bound} and is summarized in the next theorem.

\begin{Thm} \label{thm:non-iid-utility-performance}
Let $\mathbf{q}\in \mathcal{P}$ be any fixed vector satisfying $\sum_{i=1}^{n} q_{i} \leq \bar{e}$.  Under the Non i.i.d. System State model, for all $t\in\{1,2,\ldots\}$, Algorithm \ref{alg:new-alg} guarantees
\begin{align}
&\frac{1}{t}\sum_{\tau=1}^{t} \mathbb{E}[U(\mathbf{p}[\tau]; \omega[\tau])] \nonumber \\
\geq& \frac{1}{t} \sum_{\tau=1}^{t} \mathbb{E}[U(\mathbf{q}; \omega[\tau])]  - \frac{V(p^{\max})^{2}}{2t} -\frac{B}{2Vt} -  \frac{D^{2} + B}{2V} \label{eq:thm-non-iid-utility-performance-eq1}
\end{align}
where  $D$ is the constant defined in Assumption \ref{ass:basic} and $B$ is the constant defined in Lemma \ref{lm:drift}. In particular, if we choose $V = \frac{1}{\epsilon}$, then 
\begin{align*}
\frac{1}{t} \sum_{\tau=1}^{t} \mathbb{E}[U(\mathbf{p}[\tau]; \omega[\tau])] \geq \frac{1}{t} \sum_{\tau=1}^{t} \mathbb{E}[U(\mathbf{q}; \omega[\tau])] - O(\epsilon), \forall t \geq \frac{1}{\epsilon^{2}}
\end{align*}
\end{Thm}
\begin{proof}
The proof is similar to that for Theorem \ref{thm:utility-performance} and follows by summing the expression from Lemma \ref{lm:non-iid-dpp-bound} over $t$ slots and simplifying the telescoping sums.
\end{proof}

Theorem \ref{thm:non-iid-utility-performance} shows the algorithm achieves an $O(\epsilon)$ approximation when compared against any fixed power action policy that satisfies $\sum_{i=1}^n q_i \leq \overline{e}$, with convergence time $O(1/\epsilon^2)$. 
This asymptotic convergence time cannot be improved even in the  special case when $e[t]=e^{max}$ for all $t$.  Indeed, this special case removes the energy availability constraint and reduces to an (unconstrained)  
online convex optimization problem for which it is known that $O(1/\epsilon^2)$ convergence time is optimal  (see the central limit theorem argument in \cite{Hazan07ML}).

\section{Numerical Experiments}
In this section, we consider an energy harvesting wireless device transmitting over  $2$ subbands whose channel strength is represented by $s_1[t]$ and $s_2[t]$, respectively. Let $\mathcal{P} = \{\mathbf{p}: p_1 +p_2 \leq 5, p_1\geq 0, p_2\geq 0\}$. Our goal is to decide the power action $\mathbf{p}[t]$ to maximize the long term utility/throughput $\lim\limits_{T\rightarrow \infty}\frac{1}{T}\sum_{t=1}^{{T}} U(\mathbf{p}[t]; \mathbf{s}[t])$ with $U(\mathbf{p}; \mathbf{s})$ given by \eqref{eq:log-utility}.

\subsection{Scenarios with i.i.d. System States}

In this subsection, we consider system states $\omega[t] = (e[t], \mathbf{s}[t])$ that are i.i.d. generated on each slot. Let harvested energy $e[t]$ satisfy the uniform distribution over interval $[0,3]$. Let both subbands be Rayleigh fading channels where $s_1[t]$ follows the Rayleigh distribution with parameter $\sigma = 0.5$ truncated in the range $[0,4]$ and $s_2 [t]$ follows the Rayleigh distribution with parameter $\sigma = 1$ truncated in the range $[0,4]$.  

\subsubsection{Performance Verification}
By assuming the perfect knowledge of distributions, we solve the deterministic problem  \eqref{eq:deterministic-obj}-\eqref{eq:deterministic-box-cons} and obtain $U^\ast = 1.0391$. To verify the performance proven in Theorems \ref{thm:utility-performance} and \ref{thm:energy-availability}, we run Algorithm \ref{alg:new-alg} with $V\in\{5, 10, 20, 40\}$ and $E[0] = E^{\max} = Q^l+p^{\max}$. All figures in this paper are obtained by averaging $200$ independent simulation runs.  In all the simulation runs, the power actions yielded by Algorithm \ref{alg:new-alg} always satisfy the energy availability constraints. We also plot the averaged utility performance in Figure \ref{fig:large_battery}, where the $y$-axis is the running average of expected utility. Figure \ref{fig:large_battery} shows that the utility performance can approach $U^\ast$ by using larger $V$ parameter.

\begin{figure}[htbp]
\centering
   \includegraphics[width=0.465\textwidth,height=0.465\textheight,keepaspectratio=true]{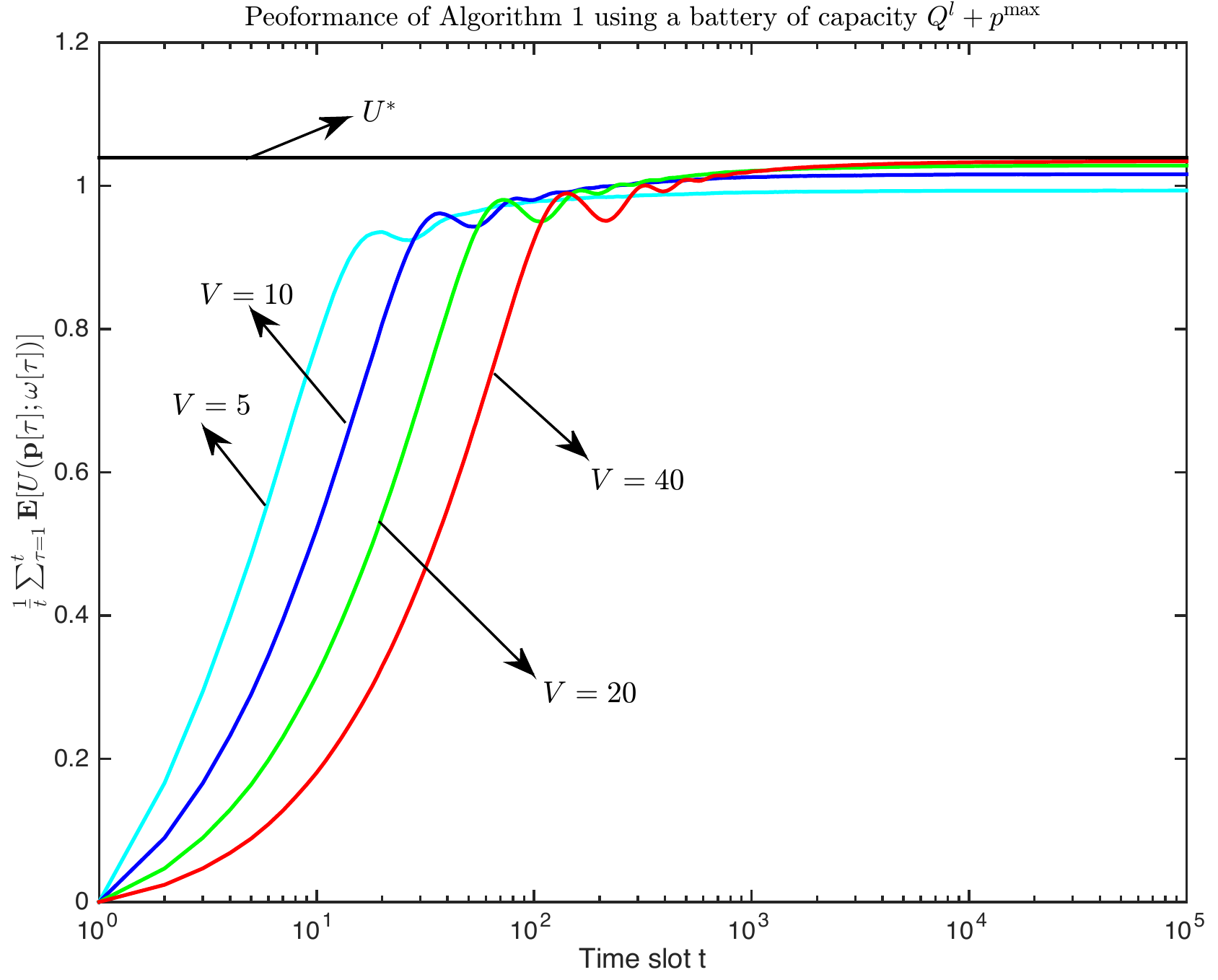} 
   \caption{Utility performance of Algorithm \ref{alg:new-alg} with $E[0] = E^{\max} = Q^l+p^{\max}$ for different $V$ in an i.i.d. system.}
   \label{fig:large_battery}
\end{figure}

\subsubsection{Performance with small battery capacity} 
In practice, it is possible that for a given $V$, the battery capacity $E^{\max} = Q^l+p^{\max}$ required in Theorem  \ref{thm:energy-availability} is too large. If we run Algorithm \ref{alg:new-alg} with small capacity batteries such that $\sum_{i=1}^n p_i[t+1]\geq E[t]$ for certain slot $t$, a reasonable choice is to scale down $\mathbf{p}[t+1]$ by \eqref{eq:explicit-energy-availability} and use $\hat{\mathbf{p}}[t+1]$ as the power action. Now, we run simulations by fixing $V = 40$ in Algorithm \ref{alg:new-alg} and test its performance with small capacity batteries.  By Theorem  \ref{thm:energy-availability}, the required battery capacity to ensure energy availability is $E^{\max} = 685$. In our simulations, we choose small $E^{\max} \in \{10, 20, 50\}$ and $E[0] = 0$, i.e., the battery is initially empty. If $\mathbf{p}[t+1]$ from Algorithm \ref{alg:new-alg} violates energy availability constraint \eqref{eq:energy-availability}, we use $\hat{\mathbf{p}}[t+1]$ from \eqref{eq:explicit-energy-availability} as the true power action that is enforced to satisfy \eqref{eq:energy-availability} and update the energy backlog by $E[t+1] = \min\{E[t] - \sum_{i=1}^n \hat{p}_i[t+1] + e[t+1], E^{\max}\}$. Figure \ref{fig:small_battery} plots the utility performance of Algorithm \ref{alg:new-alg} in this practical scenario and shows that even with small capacity batteries, Algorithm  \ref{alg:new-alg} still achieves a utility close to $U^\ast$. This further demonstrates the superior performance of our algorithm.  

\begin{figure}[htbp]
\centering
   \includegraphics[width=0.465\textwidth,height=0.465\textheight,keepaspectratio=true]{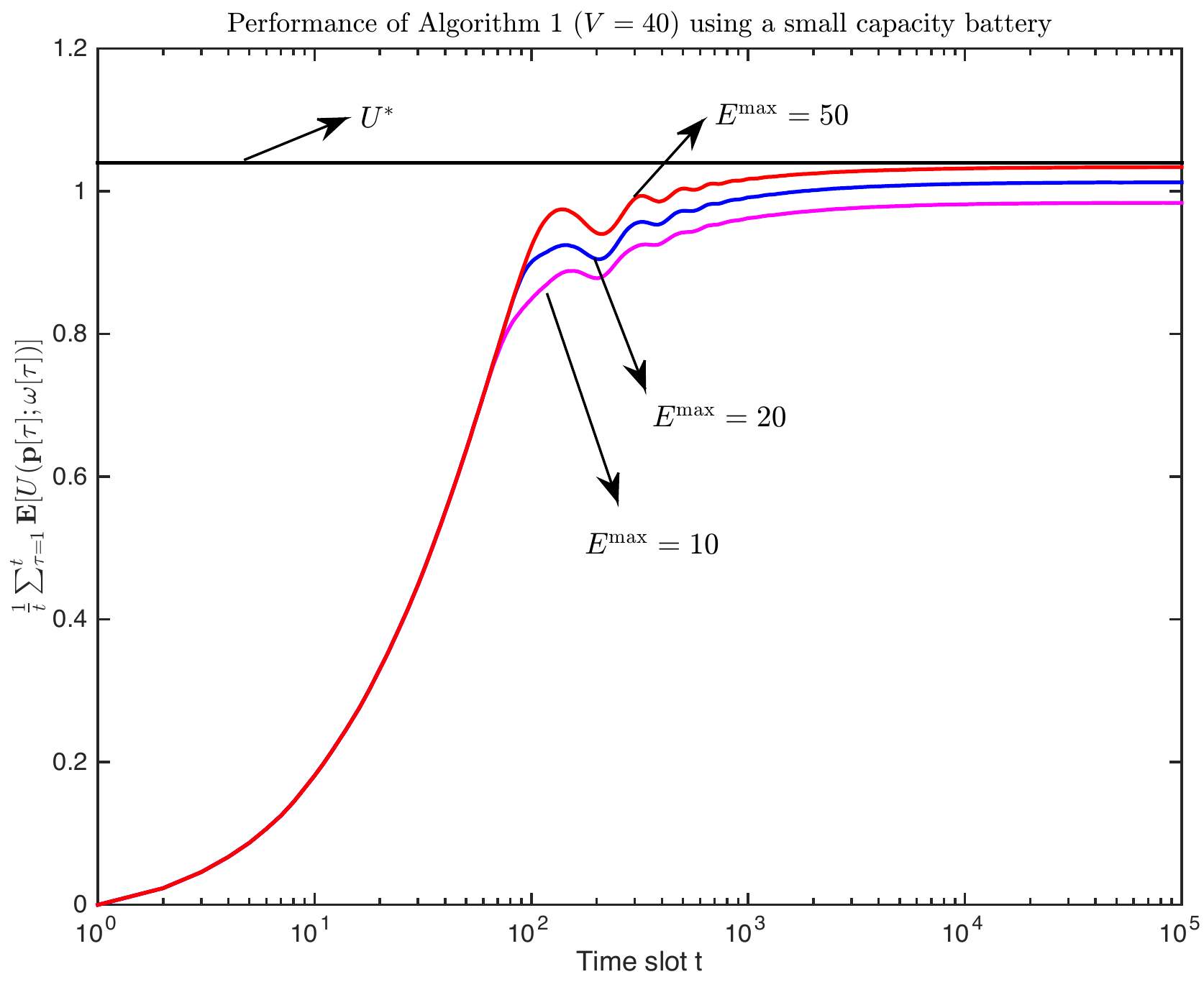} 
   \caption{Utility performance of Algorithm \ref{alg:new-alg} with $V=40$ for different $E^{\max}$ in an i.i.d. system.}
   \label{fig:small_battery}
\end{figure}

\subsubsection{Longer System State Observation Delay} Now consider the situation where the system state $\omega[t] = (e[t], \mathbf{s}[t])$ is observed with $t_{0}>1$ slot delay.  As discussed in Section \ref{sec:large-delay},  $t_{0}>1$ does not affect the $[O(\epsilon), O(1/\epsilon)]$ tradeoff established in Theorems \ref{thm:utility-performance} and \ref{thm:energy-availability}. We now run simulations to verify the effect of observation delay $t_{0}$ for our algorithm's performance. We set the battery capacity $E^{\max}=20$,  $E[0]=0$ and run Algorithm \ref{alg:new-alg} with $V=40$ using the modified updates described in Section \ref{sec:large-delay} with $t_{0}\in\{1,5,10\}$. Note that if the yielded power vector at one slot uses more than available energy in the battery, we also need to scale it down using \eqref{eq:explicit-energy-availability}. Figure \ref{fig:delay} plots the utility performance of Algorithm \ref{alg:new-alg} with different system state observation delay. As observed in the figure, a larger $t_{0}$ can slow down the convergence of our algorithm but has a negligible effect on the long term performance.
\begin{figure}[htbp]
\centering
   \includegraphics[width=0.465\textwidth,height=0.465\textheight,keepaspectratio=true]{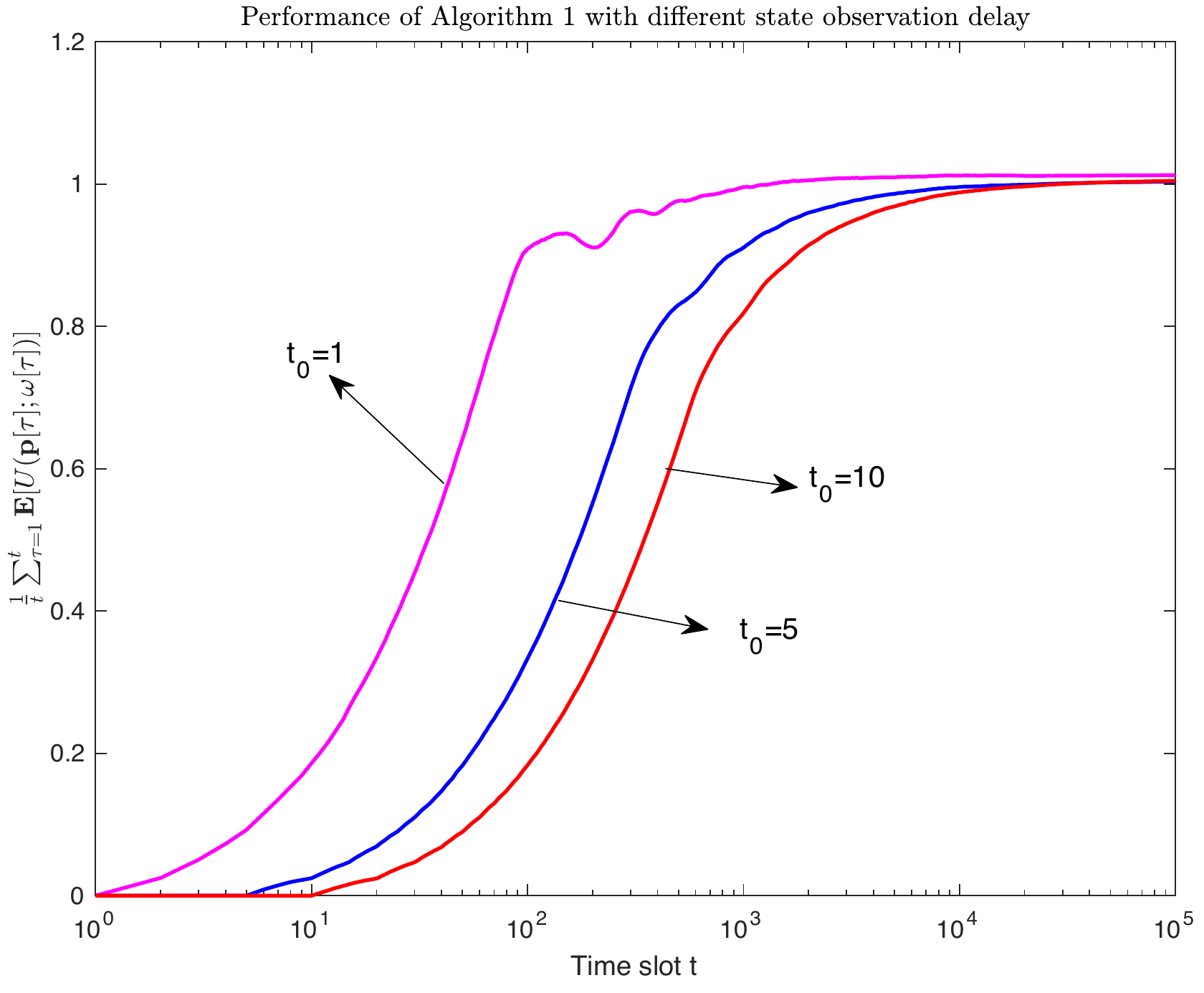} 
   \caption{Utility performance  of Algorithm \ref{alg:new-alg} for different system state observation delay  in an i.i.d. system.}
   \label{fig:delay}
\end{figure}

\subsubsection{Comparison with Other Schemes}

As reviewed in Section \ref{sec:alg-intuition}, the conventional Zinkevich's online convex learning \eqref{eq:Zinkevich} with $\gamma =1/V$ is similar to the power control step \eqref{eq:power-update} in Algorithm \ref{alg:new-alg} except that term $\frac{1}{V^{2}}Q[t]\mathbf{1}$ is dropped. The weakness of \eqref{eq:Zinkevich} is that its yielded power actions can violate the energy availability constraint \eqref{eq:energy-availability}. Now consider the scheme that yields power actions as follows:
\begin{align}
\mathbf{p}[t+1] = \mbox{Proj}_{\mathcal{P}(t)} \{ \mathbf{p}[t] + \gamma \nabla_\mathbf{p} U(\mathbf{p}[t]; \omega[t]) \}
\end{align}
with $\mathcal{P}(t) = \{\mathbf{p}\in \mathbb{R}^n: \sum_{i=1}^n p_i \leq  \min\{p^{\max}, E[t]\}, p_i \geq 0, \forall i\in
\{1,2,\ldots,n\}\}$.  This scheme is a simple modification of Zinkevich's online convex learning \eqref{eq:Zinkevich} to ensure \eqref{eq:energy-availability} by projecting onto time-varying sets $\mathcal{P}(t)$ that restrict the total used power to be no more than current energy buffer $E[t]$. We call this scheme Baseline 1.  We also consider another scheme that chooses $\mathbf{p}[t+1]$ as the solution to the following optimization:
\begin{align*}
\mathbf{p}[t+1] = \mbox{argmax}_{\mathcal{P}(t)} \{ U(\mathbf{p}; \omega[t])\}
\end{align*} 
Similar to Baseline 1, this scheme can ensure the energy availability by restricting its power  vector to the time-varying sets $\mathcal{P}(t)$.  Different from Baseline 1, power vector $\mathbf{p}[t+1]$ is chosen to maximize an outdated utility $U(\mathbf{p}; \omega[t])$. We call this scheme Baseline 2.    We set the battery capacity $E^{{\max}}=10$,  $E[0]=0$; and run Algorithm \ref{alg:new-alg} with $V=50$ (with power vector scaled down when necessary), Baseline 1 with $\gamma =1/V$ and Baseline 2. Figure \ref{fig:compare} plots the utility performance of all schemes and demonstrates that Algorithm \ref{alg:new-alg} has the best performance.
\begin{figure}[htbp]
\centering
   \includegraphics[width=0.465\textwidth,height=0.465\textheight,keepaspectratio=true]{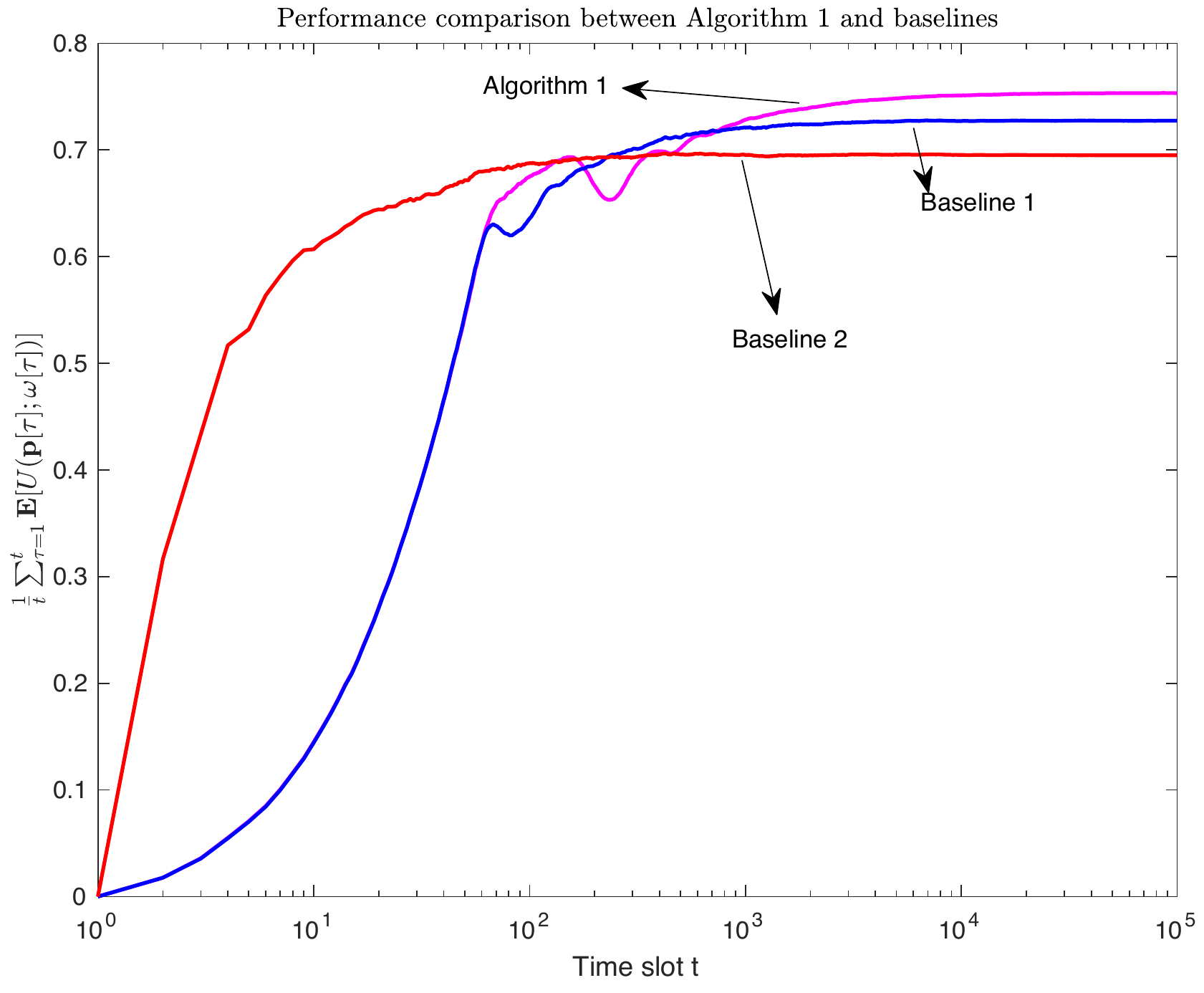} 
   \caption{Performance comparison between Algorithm \ref{alg:new-alg} and two baselines  in an i.i.d. system.}
   \label{fig:compare}
\end{figure}

\subsection{Scenarios with Non i.i.d. System States}
Now consider system states $\omega[t] = (e[t], \mathbf{s}[t])$ that are not i.i.d. generated on each slot. Assume the subband channels $\mathbf{s}[t] = (s_{1}[t], s_{2}[t])$ evolve according to a $2$-state Markov chain. The first state of the Markov chain is $(s_{1} = 0.45, s_{2}=1.2)$ and the second state is $(s_{1} = 1, s_{2}=0.2)$.  The transition probability of the Markov chain is $ \left[\begin{array}{cc}  1/15  & 14/15 \\ 2/3  &1/3\end{array}\right]$ where the $(i,j)$-th entry denotes the  Markov chain's transition probability from state $i$ to state $j$.  The harvested energy $e[t]$ is still i.i.d. generated from the uniform distribution over interval $[0,3]$.

We repeat the same experiments as we did in Figures \ref{fig:large_battery}-\ref{fig:compare}. The only difference is the channel subbands are now time-correlated and evolve according to the above $2$-state Markov chain. The observations from Figures \ref{fig:markov_large_battery}-\ref{fig:markov_compare} for such a non-i.i.d. system are consistent with observations in the i.i.d. case.

\begin{figure}[htbp]
\centering
   \includegraphics[width=0.465\textwidth,height=0.465\textheight,keepaspectratio=true]{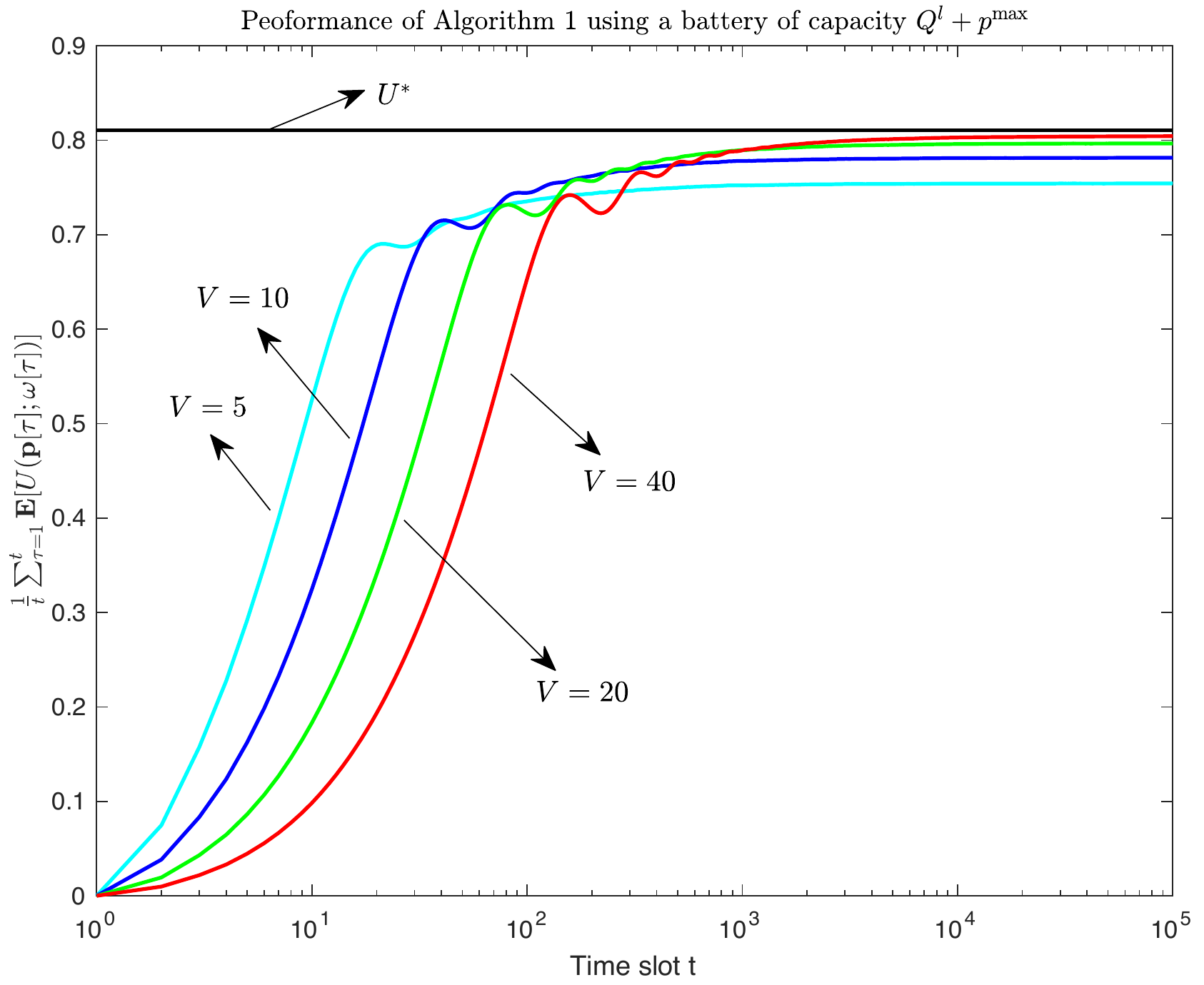} 
   \caption{Utility performance of Algorithm \ref{alg:new-alg} with $E[0] = E^{\max} = Q^l+p^{\max}$ for different $V$ in a non-i.i.d. system.}
   \label{fig:markov_large_battery}
\end{figure}

\begin{figure}[htbp]
\centering
   \includegraphics[width=0.465\textwidth,height=0.465\textheight,keepaspectratio=true]{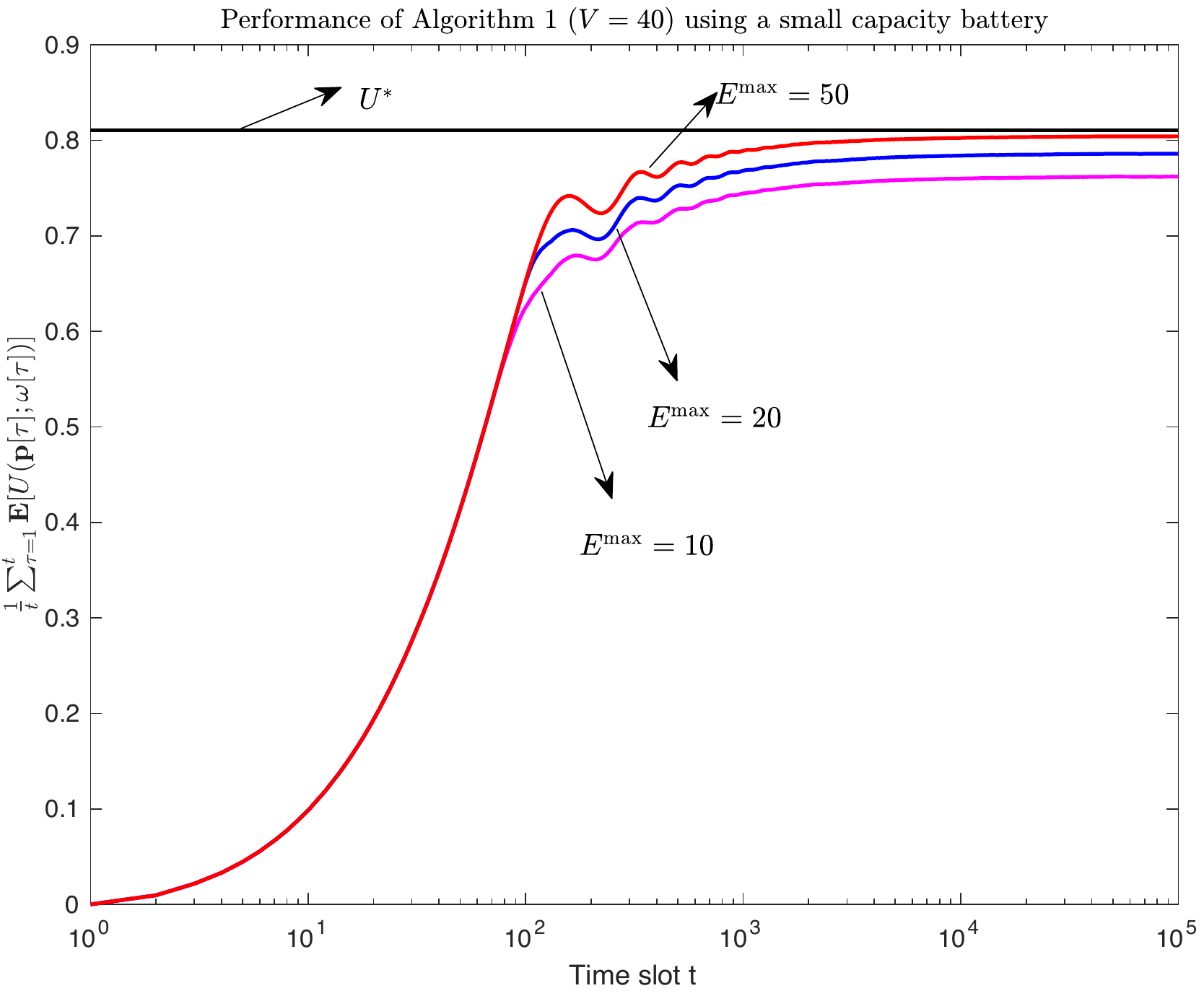} 
   \caption{Utility performance of Algorithm \ref{alg:new-alg} with $V=40$ for different $E^{\max}$ in a non-i.i.d. system.}
   \label{fig:markov_small_battery}
\end{figure}

\begin{figure}[htbp]
\centering
   \includegraphics[width=0.465\textwidth,height=0.465\textheight,keepaspectratio=true]{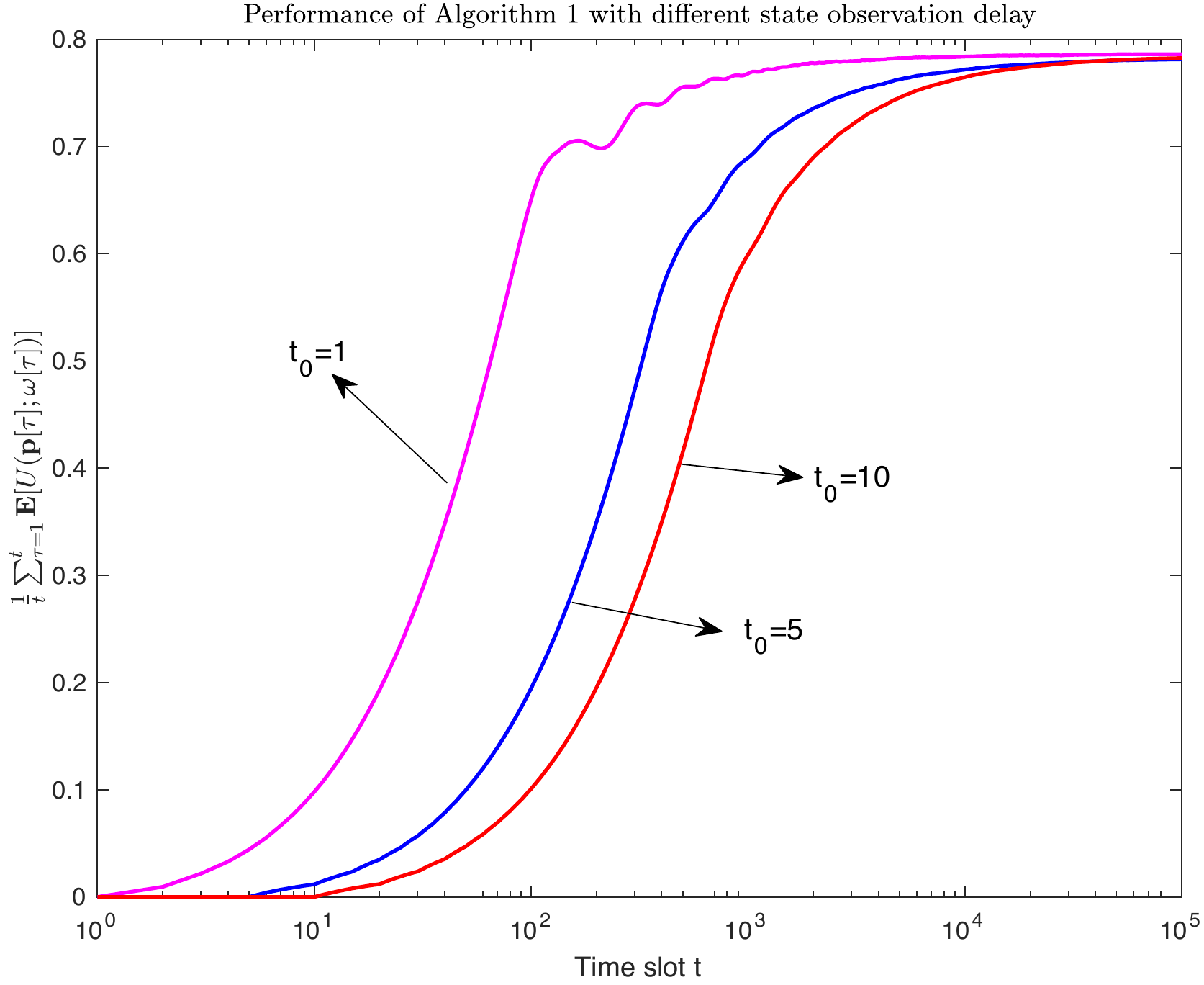} 
   \caption{Utility performance  of Algorithm \ref{alg:new-alg} for different system state observation delay  in a non-i.i.d. system.}
   \label{fig:markov_delay}
\end{figure}

\begin{figure}[htbp]
\centering
   \includegraphics[width=0.465\textwidth,height=0.465\textheight,keepaspectratio=true]{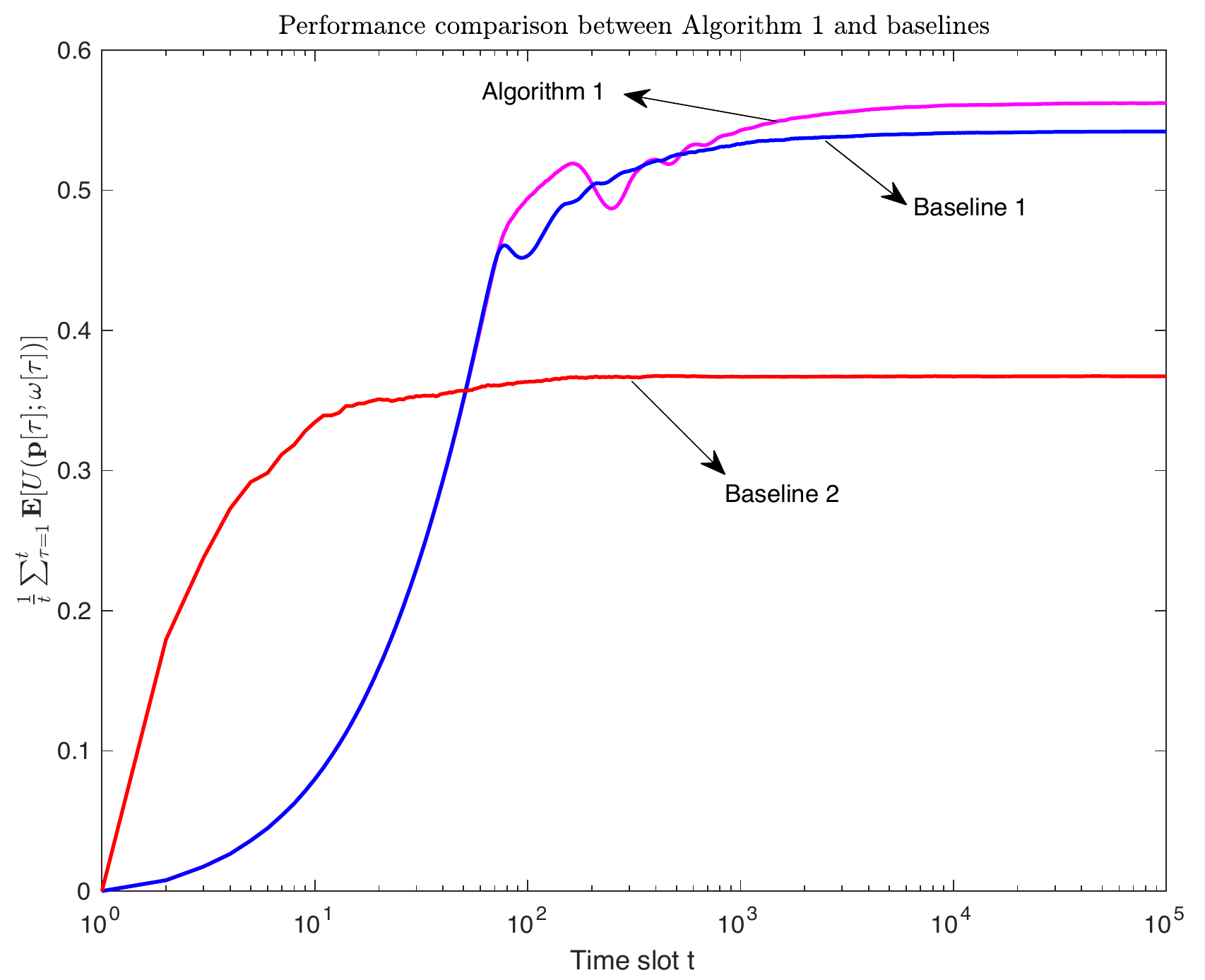} 
   \caption{Performance comparison between Algorithm \ref{alg:new-alg} and two baselines  in a non-i.i.d. system.}
   \label{fig:markov_compare}
\end{figure}

\section{Conclusion}
This paper develops a new learning aided power control algorithm for energy harvesting devices, without requiring the current system state or the distribution information.  This new algorithm can achieve an $O(\epsilon)$ optimal utility by using a battery with capacity $O(1/\epsilon)$ and a  convergence time of $O(1/\epsilon^2)$.  

\bibliographystyle{IEEEtran}
\bibliography{mybibfile}

\begin{thebibliography}{10}
\providecommand{\url}[1]{#1}
\csname url@samestyle\endcsname
\providecommand{\newblock}{\relax}
\providecommand{\bibinfo}[2]{#2}
\providecommand{\BIBentrySTDinterwordspacing}{\spaceskip=0pt\relax}
\providecommand{\BIBentryALTinterwordstretchfactor}{4}
\providecommand{\BIBentryALTinterwordspacing}{\spaceskip=\fontdimen2\font plus
\BIBentryALTinterwordstretchfactor\fontdimen3\font minus
  \fontdimen4\font\relax}
\providecommand{\BIBforeignlanguage}[2]{{%
\expandafter\ifx\csname l@#1\endcsname\relax
\typeout{** WARNING: IEEEtran.bst: No hyphenation pattern has been}%
\typeout{** loaded for the language `#1'. Using the pattern for}%
\typeout{** the default language instead.}%
\else
\language=\csname l@#1\endcsname
\fi
#2}}
\providecommand{\BIBdecl}{\relax}
\BIBdecl

\bibitem{YuNeely18INFOCOM}
H.~Yu and M.~J. Neely, ``Learning aided optimization for energy harvesting
  devices with outdated state information,'' in \emph{Proceedings of IEEE
  International Conference on Computer Communications (INFOCOM)}, 2018.

\bibitem{Paradiso05}
J.~A. Paradiso and T.~Starner, ``Energy scavenging for mobile and wireless
  electronics,'' \emph{IEEE Pervasive Computing}, vol.~4, no.~1, pp. 18--27,
  2005.

\bibitem{Sudevalayam11}
S.~Sudevalayam and P.~Kulkarni, ``Energy harvesting sensor nodes: Survey and
  implications,'' \emph{IEEE Communications Surveys \& Tutorials}, vol.~13,
  no.~3, pp. 443--461, 2011.

\bibitem{Ulukus15JSAC}
S.~Ulukus, A.~Yener, E.~Erkip, O.~Simeone, M.~Zorzi, P.~Grover, and K.~Huang,
  ``Energy harvesting wireless communications: A review of recent advances,''
  \emph{IEEE Journal on Selected Areas in Communications}, vol.~33, no.~3, pp.
  360--381, 2015.

\bibitem{Kansal07TECS}
A.~Kansal, J.~Hsu, S.~Zahedi, and M.~B. Srivastava, ``Power management in
  energy harvesting sensor networks,'' \emph{ACM Transactions on Embedded
  Computing Systems}, vol.~6, no.~4, 2007.

\bibitem{Kamalinejad15}
P.~Kamalinejad, C.~Mahapatra, Z.~Sheng, S.~Mirabbasi, V.~C. Leung, and Y.~L.
  Guan, ``Wireless energy harvesting for the internet of things,'' \emph{IEEE
  Communications Magazine}, vol.~53, no.~6, pp. 102--108, 2015.

\bibitem{Hossian15}
E.~Hossain and M.~Hasan, ``{5G} cellular: key enabling technologies and
  research challenges,'' \emph{IEEE Instrumentation \& Measurement Magazine},
  vol.~18, no.~3, pp. 11--21, 2015.

\bibitem{Ulukus12TCOM}
J.~Yang and S.~Ulukus, ``Optimal packet scheduling in an energy harvesting
  communication system,'' \emph{IEEE Transactions on Communications}, vol.~60,
  no.~1, pp. 220--230, 2012.

\bibitem{Yener12TWC}
K.~Tutuncuoglu and A.~Yener, ``Optimum transmission policies for battery
  limited energy harvesting nodes,'' \emph{IEEE Transactions on Wireless
  Communications}, vol.~11, no.~3, pp. 1180--1189, 2012.

\bibitem{Blasco13TWC}
P.~Blasco, D.~Gunduz, and M.~Dohler, ``A learning theoretic approach to energy
  harvesting communication system optimization,'' \emph{IEEE Transactions on
  Wireless Communications}, vol.~12, no.~4, pp. 1872--1882, 2013.

\bibitem{Michelusi13TCOM}
N.~Michelusi, K.~Stamatiou, and M.~Zorzi, ``Transmission policies for energy
  harvesting sensors with time-correlated energy supply,'' \emph{IEEE
  Transactions on Communications}, vol.~61, no.~7, pp. 2988--3001, 2013.

\bibitem{Shaviv16JSAC}
D.~Shaviv and A.~{\"O}zg{\"u}r, ``Universally near optimal online power control
  for energy harvesting nodes,'' \emph{IEEE Journal on Selected Areas in
  Communications}, vol.~34, no.~12, pp. 3620--3631, 2016.

\bibitem{Wu17TMC}
W.~Wu, J.~Wang, X.~Wang, F.~Shan, and J.~Luo, ``Online throughput maximization
  for energy harvesting communication systems with battery overflow,''
  \emph{IEEE Transactions on Mobile Computing}, vol.~16, no.~1, pp. 185--197,
  2017.

\bibitem{Arafa17ISIT}
A.~Arafa, A.~Baknina, and S.~Ulukus, ``Energy harvesting networks with general
  utility functions: Near optimal online policies,'' in \emph{IEEE
  International Symposium on Information Theory (ISIT)}, 2017, pp. 809--813.

\bibitem{Gatzianas10TWC}
M.~Gatzianas, L.~Georgiadis, and L.~Tassiulas, ``Control of wireless networks
  with rechargeable batteries,'' \emph{IEEE Transactions on Wireless
  Communications}, vol.~9, no.~2, pp. 581--593, 2010.

\bibitem{HuangNeely13TON}
L.~Huang and M.~J. Neely, ``Utility optimal scheduling in energy-harvesting
  networks,'' \emph{IEEE/ACM Transactions on Networking}, vol.~21, no.~4, pp.
  1117--1130, 2013.

\bibitem{UrgaonkarNeely11Sigmetrics}
R.~Urgaonkar, B.~Urgaonkar, M.~J. Neely, and A.~Sivasubramaniam, ``Optimal
  power cost management using stored energy in data centers,''
  \emph{Proceedings of ACM SIGMETRICS}, 2011.

\bibitem{Zinkevich03ICML}
M.~Zinkevich, ``Online convex programming and generalized infinitesimal
  gradient ascent,'' in \emph{Proceedings of International Conference on
  Machine Learning (ICML)}, 2003.

\bibitem{book_PredictionLearningGames}
N.~Cesa-Bianchi and G.~Lugosi, \emph{Prediction, Learning, and Games}.\hskip
  1em plus 0.5em minus 0.4em\relax Cambridge University Press, 2006.

\bibitem{Shalev-Shwartz11FoundationTrends}
S.~Shalev-Shwartz, ``Online learning and online convex optimization,''
  \emph{Foundations and Trends in Machine Learning}, vol.~4, no.~2, pp.
  107--194, 2011.

\bibitem{Mahdavi12JMLR}
M.~Mahdavi, R.~Jin, and T.~Yang, ``Trading regret for efficiency: online convex
  optimization with long term constraints,'' \emph{Journal of Machine Learning
  Research}, vol.~13, no.~1, pp. 2503--2528, 2012.

\bibitem{Jenatton16ICML}
R.~Jenatton, J.~Huang, and C.~Archambeau, ``Adaptive algorithms for online
  convex optimization with long-term constraints,'' in \emph{Proceedings of
  International Conference on Machine Learning (ICML)}, 2016.

\bibitem{YuNeely16ArxivOnlineOpt}
H.~Yu and M.~J. Neely, ``A low complexity algorithm with ${O}(\sqrt{T})$ regret
  and finite constraint violations for online convex optimization with long
  term constraints,'' \emph{arXiv:1604.02218}, 2016.

\bibitem{NeelyYu17arXiv}
M.~J. Neely and H.~Yu, ``Online convex optimization with time-varying
  constraints,'' \emph{arXiv:1702.04783}, 2017.

\bibitem{YuNeely17NIPS}
H.~Yu, M.~J. Neely, and X.~Wei, ``Online convex optimization with stochastic
  constraints,'' in \emph{Advances in Neural Information Processing Systems},
  2017, pp. 1427--1437.

\bibitem{book_Neely10}
M.~J. Neely, \emph{Stochastic Network Optimization with Application to
  Communication and Queueing Systems}.\hskip 1em plus 0.5em minus 0.4em\relax
  Morgan \& Claypool Publishers, 2010.

\bibitem{Srivastava13TON}
R.~Srivastava and C.~E. Koksal, ``Basic performance limits and tradeoffs in
  energy-harvesting sensor nodes with finite data and energy storage,''
  \emph{IEEE/ACM Transactions on Networking (TON)}, vol.~21, no.~4, pp.
  1049--1062, 2013.

\bibitem{YuNeely17INFOCOM}
H.~Yu and M.~J. Neely, ``A new backpressure algorithm for joint rate control
  and routing with vanishing utility optimality gaps and finite queue
  lengths,'' in \emph{Proceedings of IEEE International Conference on Computer
  Communications (INFOCOM)}, 2017.

\bibitem{Tassiulas92TAC}
L.~Tassiulas and A.~Ephremides, ``Stability properties of constrained queueing
  systems and scheduling policies for maximum throughput in multihop radio
  networks,'' \emph{IEEE Transactions on Automatic Control}, vol.~37, no.~12,
  pp. 1936--1948, 1992.

\bibitem{Neely05JSAC}
M.~J. Neely, E.~Modiano, and C.~E. Rohrs, ``Dynamic power allocation and
  routing for time-varying wireless networks,'' \emph{IEEE Journal on Selected
  Areas in Communications}, vol.~23, no.~1, pp. 89--103, 2005.

\bibitem{Eryilmaz06JSAC}
A.~Eryilmaz and R.~Srikant, ``Joint congestion control, routing, and mac for
  stability and fairness in wireless networks,'' \emph{IEEE Journal on Selected
  Areas in Communications}, vol.~24, no.~8, pp. 1514--1524, 2006.

\bibitem{Stolyar05}
A.~L. Stolyar, ``Maximizing queueing network utility subject to stability:
  Greedy primal-dual algorithm,'' \emph{Queueing Systems}, vol.~50, no.~4, pp.
  401--457, 2005.

\bibitem{Neely08TON}
M.~J. Neely, E.~Modiano, and C.-P. Li, ``Fairness and optimal stochastic
  control for heterogeneous networks,'' \emph{IEEE/ACM Transactions on
  Networking}, vol.~16, no.~2, pp. 396--409, 2008.

\bibitem{Hazan07ML}
E.~Hazan, A.~Agarwal, and S.~Kale, ``Logarithmic regret algorithms for online
  convex optimization,'' \emph{Machine Learning}, vol.~69, pp. 169--192, 2007.

\bibitem{YuNeely17SIOPT}
H.~Yu and M.~J. Neely, ``A simple parallel algorithm with an ${O}(1/t)$
  convergence rate for general convex programs,'' \emph{SIAM Journal on
  Optimization}, vol.~27, no.~2, pp. 759--783, 2017.

\bibitem{Mannor09JMLR}
S.~Mannor, J.~N. Tsitsiklis, and J.~Y. Yu, ``Online learning with sample path
  constraints,'' \emph{Journal of Machine Learning Research}, vol.~10, pp.
  569--590, March 2009.

\end{thebibliography}

\end{document}